\documentclass[oneside,a4paper]{amsart}

\usepackage[T1]{fontenc}
\usepackage[utf8]{inputenc}
\usepackage[british]{babel}
\usepackage{amsmath, amsthm, amsfonts, amssymb}

\usepackage{hyperref}
\usepackage{url}
\usepackage[all]{xy}
\usepackage{placeins} % barriers for floats
\usepackage[usenames, dvipsnames]{color} % colourfultext
\usepackage{euscript} % Lurie-type Euler script
\usepackage{bbm} % mathbb numbers
\usepackage{enumitem} % fancy enumerate
\usepackage{relsize}

\usepackage[usenames,dvipsnames]{color}
\usepackage{tikz}
\usetikzlibrary{arrows,backgrounds,
    decorations.pathreplacing,
    decorations.pathmorphing
}
\usetikzlibrary{matrix,arrows}

\newcommand{\euscr}[1]{\EuScript{#1}} % Euler script
\newcommand{\ccat}{\euscr{C}} % category C in Euler script 
\newcommand{\dcat}{\euscr{D}} % category D in Euler script
\newcommand{\Fun}{\textnormal{Fun}} % functor category
\newcommand{\Hom}{\textnormal{Hom}} % homomorphims 
\newcommand{\Ext}{\textnormal{Ext}} % Ext 
\newcommand{\map}{\textnormal{map}} % mapping space
\newcommand{\spaces}{\euscr{S}} % the category of spaces
\newcommand{\spectra}{\euscr{S}p} % the category of spectra
\newcommand{\synspectra}{\euscr{S}yn} % synthetic spectra
\newcommand{\ComodE}{\euscr{C}omod_{E_{*}E}} % E_*E-comodules
\newcommand{\Mod}{\euscr{M}od} % module infinity-category
\newcommand{\monunit}{\mathbbm{1}} % the monoidal unit
\newcommand{\monunitt}[1]{\monunit_{\leq #1}} %the truncated monoital unit
\theoremstyle{plain}

\newtheorem{thm}{Theorem}[section]
\newtheorem{lemma}[thm]{Lemma}
\newtheorem{prop}[thm]{Proposition}
\newtheorem{cor}[thm]{Corollary}

\newtheorem*{thm*}{Theorem}

\theoremstyle{definition}
\newtheorem{example}[thm]{Example}
\newtheorem{warning}[thm]{Warning}

\newtheorem{defin}[thm]{Definition}
\newtheorem{rem}[thm]{Remark}
\newtheorem{notation}[thm]{Notation}

\newtheorem{assumption}[thm]{Assumption}
\newtheorem*{rem*}{Remark}
\newtheorem*{interpretation*}{Interpretation}
\newtheorem*{defin*}{Definition}
\newtheorem*{conjecture*}{Conjecture}
\newtheorem*{notation*}{Notation}
\newtheorem*{convention*}{Convention}
\newtheorem*{thm_italics*}{Theorem}

\theoremstyle{remark}

\makeatletter
  \def\subsection{\@startsection{subsection}{1}%
  \z@{.7\linespacing\@plus\linespacing}{.5\linespacing}%
  {\normalfont\bfseries\centering}}% NEW
\makeatother

\let\oldtocsection=\tocsection
\let\oldtocsubsection=\tocsubsection
\let\oldtocsubsubsection=\tocsubsubsection
\renewcommand{\tocsection}[2]{\hspace{0em}\oldtocsection{#1}{#2}}
\renewcommand{\tocsubsection}[2]{\hspace{1em}\oldtocsubsection{#1}{#2}}
\renewcommand{\tocsubsubsection}[2]{\hspace{2em}\oldtocsubsubsection{#1}{#2}}

\calclayout

\begin{document}
\title[Chromatic homotopy theory is algebraic when $p > n^{2}+n+1$]{Chromatic homotopy theory is algebraic when $p > n^{2}+n+1$}
\author[Piotr Pstr\k{a}gowski]{Piotr Pstr\k{a}gowski}
\address{Northwestern University}
\email{pstragowski.piotr@gmail.com}

\begin{abstract}
We show that if $E$ is a $p$-local Landweber exact homology theory of height $n$ and $p  > n^2+n+1$, then there exists an equivalence $h\spectra_{E} \simeq h\dcat(E_{*}E)$ between homotopy categories of $E$-local spectra and differential $E_{*}E$-comodules, generalizing Bousfield's and Franke's results to heights $n > 1$. 
\end{abstract}

\maketitle 

\tableofcontents

% the beginning sections do not go into the table of contents 
% \addtocontents{toc}{\protect\setcounter{tocdepth}{-1}}

\section{Introduction}

In chromatic homotopy theory one studies the connection between stable homotopy and the theory of formal groups. Locally at a prime $p$, the moduli stack of formal groups admits a filtration by open substacks classifying formal groups of height at most $n$; in homotopy theory this is reflected by a sequence of localization functors $L_{n}: \spectra \rightarrow \spectra$, where $L_{n}$ is the Bousfield localization with respect to a $p$-local Landweber exact homology theory of height $n$. 

As a consequence of the chromatic convergence theorem of Ravenel and Hopkins, if $X$ is a finite spectrum, then $\varprojlim L_{n} X$ is equivalent to the $p$-localization $X_{(p)}$ \cite{ravenel2016nilpotence}. In principle, this implies that $\pi_{*}X$ can be computed from the knowledge of $\pi_{*} L_{n}X$, which is something one can approach inductively using the chromatic fracture squares. 

From a different perspective, chromatic convergence also implies that as height grows, the homotopy groups $\pi_{*} L_{n} X$ must become at least as complicated as $\pi_{*} X$, where by complicated we mean \emph{fundamentally non-algebraic}. This is a consequence of the Mahowald uncertainty principle, which informally states that any algebraic approximation to the homotopy groups of spheres must be infinitely far away from the actual answer \cite{goerss2008adams}. 

On the other hand, if $E$ is a Landweber exact homology theory of small height, then the $E$-local homotopy theory is known to simplify considerably. For example, when $2p - 2 > n^{2}+n$, then the $E$-local Adams-Novikov spectral sequence collapses for degree reasons and induces an isomorphism between $\pi_{*} S^{0}_{E}$ and $\Ext_{E_{*}E}(E_{*}, E_{*})$. 

Similarly, it is a result of Hovey and Sadofsky that under the same assumption, the only invertible $E$-local spectra are the spheres \cite{hovey1999invertible}, in stark contrast to what happens when the prime is small \cite{goerss2014hopkins}, \cite{hovey1999invertible}. 

In this note, we prove a global statement on the algebraicity of chromatic homotopy at large primes. Namely, we show that if $p > n^{2}+n+1$, then the homotopy category of $E$-local spectra is equivalent to the homotopy category of differential $E_{*}E$-comodules. 

\subsection{Statement of results}

Let $E$ be $p$-local multiplicative Landweber exact homology theory, so that $E_{*}$ is a $BP_{*}$-algebra. Following Strickland and Hovey, we say that $E$ is of \emph{height $n$} if $E_{*} / I_{n} \neq 0$ and $E_{*} / I_{n+1} = 0$, where  $I_{n} = (v_{0}, \ldots, v_{n-1})$. For example, $E$ could be the Lubin-Tate spectrum $E(k, \Gamma)$ associated to a formal group $\Gamma / k$ of height $n$, or the Johnson-Wilson theory $E(n)$. 

A \emph{differential $E_{*}E$-comodule} is pair $(M, d)$, where $M$ is an $E_{*}E$-comodule and $d: M \rightarrow M$ is a map of comodules of degree $1$ satisfying $d^{2} = 0$. The collection of differential comodules has a natural notion of weak equivalence given by quasi-isomorphisms, by inverting those we obtain the \emph{homotopy category}, which we will denote by $h\dcat(E_{*}E)$. 

\begin{thm}[\ref{thm:at_large_primes_homotopy_category_of_e_local_spectra_same_as_periodic_chain_complexes}]
\label{thm:intro_homotopy_kcategories_of_elocal_spectra_and_periodic_chain_complexes_equivalent_at_large_primes}
Let $E$ be a $p$-local Landweber exact homology theory of height $n$ and assume that $p > n^{2} + n + 1$. Then, there exists an equivalence $h\spectra_{E} \simeq h\mathcal{D}(E_{*}E)$ between the homotopy categories of $E$-local spectra and differential $E_{*}E$-comodules. 
\end{thm}
In particular, \textbf{Theorem \ref{thm:intro_homotopy_kcategories_of_elocal_spectra_and_periodic_chain_complexes_equivalent_at_large_primes}} yields a complete classification of homotopy types of $E$-local spectra, and of homotopy classes of maps between them. It also implies that any $E_{*}E$-comodule has a canonical realization as a homology of an $E$-local spectrum; as we show in a companion paper, this alone already forces algebraicity of the Picard group of the corresponding $K(n)$-local category \cite{chromatic_picard_groups_at_large_primes}. 

At $n = 1$, the existence of an equivalence $h \spectra_{E} \simeq h \dcat(E_{*}E)$ is a classical result of Bousfield \cite{bousfield1985homotopy}. A generalization to all heights was attempted by Franke \cite{franke1996uniqueness}, but a subtle error in the latter work was found by Patchkoria \cite{patchkoria2017exotic}. 

The equivalence of \textbf{Theorem \ref{thm:intro_homotopy_kcategories_of_elocal_spectra_and_periodic_chain_complexes_equivalent_at_large_primes}} gets stronger the larger the prime in a way which we now make precise. Observe that the collection of $E$-local spectra can be naturally assembled into the stable $\infty$-category $\spectra_{E}$. The same can be said of differential $E_{*}E$-comodules, whose category admits a model structure from which one can extract the underlying $\infty$-category, which we denote by $\mathcal{D}(E_{*}E$) and call the \emph{derived $\infty$-category of $E_{*}E$}.

\begin{thm}[\ref{thm:at_large_primes_homotopy_category_of_e_local_spectra_same_as_periodic_chain_complexes}]
\label{thm:intro_homotopy_kcategories_of_elocal_spectra_and_periodic_chain_complexes_equivalent_at_large_primes_using_k_categories}
Let $E$ be a $p$-local Landweber exact homology theory of height $n$ and assume that $p > n^{2} + n + 1 + \frac{k}{2}$. Then, there exists an equivalence $h_{k}\spectra_{E} \simeq h_{k}\dcat(E_{*}E)$ between the homotopy $k$-categories of the $\infty$-category of $E$-local spectra and the derived $\infty$-category of $E_{*}E$.
\end{thm}

Here, by a homotopy $k$-category we mean the $\infty$-category with the same objects and where the mapping spaces have been $(k-1)$-truncated \cite{lurie_higher_topos_theory}[2.3.4.12]. In particular, if $k = 1$, then the above result reduces to \textbf{Theorem \ref{thm:intro_homotopy_kcategories_of_elocal_spectra_and_periodic_chain_complexes_equivalent_at_large_primes}}. 

The higher homotopy categories carry strictly more information; for example, we show in the appendix that the triangulated structure on the homotopy category of a stable $\infty$-category is already determined by the homotopy $2$-category, see \textbf{Theorem \ref{thm:triangulated_structure_on_the_homotopy_cat_of_stable_cat_only_depends_on_homotopy_2_cat}}. Thus, if $p > n^{2}+n+2$, then the equivalence $h\spectra_{E} \simeq h\dcat(E_{*}E)$ is triangulated. 

Note that it is necessary to pass to homotopy categories, because the stable $\infty$-categories $\spectra_{E}$ and $\dcat(E_{*}E)$ are \emph{not} equivalent at any prime \cite{barthel2017chromatic}[5.32]. Thus, the equivalence of \textbf{Theorem \ref{thm:intro_homotopy_kcategories_of_elocal_spectra_and_periodic_chain_complexes_equivalent_at_large_primes}} is what is classically known as an exotic equivalence, one that doesn't preserve all of the the higher homotopical information. 

To the author's knowledge, this is a first example of an exotic equivalence where the homological dimension of the relevant abelian category is greater then $3$, the latter achieved in \cite{patchkoria2017exotic}. In fact, in our case it is $n^{2}+n$, the dimension of $E_{*}E$, and so it is unbounded. 

We should mention the result which inspired the current work, namely the asymptotic algebraicity of Barthel, Schlank and Stapleton, who prove that for any non-principal ultrafilter $\mathcal{F}$ on the set of primes, the ultraproducts $\prod _{\mathcal{F}} \spectra_{E} \simeq \prod_{\mathcal{F}} \dcat(E_{*}E)$ are equivalent as symmetric monoidal $\infty$-categories  \cite{barthel2017chromatic}. 

Informally, asymptotic algebraicity shows that "in the limit $p \mapsto \infty$" the $\infty$-categories $\spectra_{E}$ and $\dcat(E_{*}E)$ are equivalent in the strongest possible sense. Our results complement this statement, providing information at finite primes, but at a cost of giving a weaker form of equivalence. 

\subsection{Overview of the approach}

As observed in the introduction, no equivalence $h\spectra_{E} \simeq h\dcat(E_{*}E)$ can be induced by an exact functor between stable $\infty$-categories $\spectra_{E}$ and $\dcat(E_{*}E)$, since such a functor would necessarily be an equivalence itself.

Informally, our approach to the proof of \textbf{Theorem \ref{thm:intro_homotopy_kcategories_of_elocal_spectra_and_periodic_chain_complexes_equivalent_at_large_primes}} is as follows: 

\begin{enumerate}
\item construct "simplified versions" of the $E$-local and derived $\infty$-categories, together with functors $\spectra_{E} \rightarrow \spectra_{E}^{\Diamond}$ and $\dcat(E_{*}Es) \rightarrow \dcat(E_{*}E)^{\Diamond}$,
\item show that the constructed functors induce equivalences $h\spectra_{E} \simeq h\spectra_{E}^{\Diamond}$ and \\ $h\dcat(E_{*}E) \simeq h\dcat(E_{*}E)^{\Diamond}$ between  homotopy categories and
\item show that the $\infty$-categories $\spectra_{E}^{\Diamond}$ and $\dcat(E_{*}E)^{\Diamond}$ are equivalent.
\end{enumerate}
In other words, rather than trying to construct an appropriate functor $\spectra_{E} \rightarrow \dcat(E_{*}E)$, which we know is impossible, we construct a zig-zag of functors which induce equivalences between homotopy categories. 

To construct the $\infty$-categories $\spectra_{E}^{\Diamond}$ and $\dcat(E_{*}E)^{\Diamond}$, the "simplified versions" of the $E$-local and derived $\infty$-categories, we use Goerss-Hopkins theory \cite{moduli_spaces_of_commutative_ring_spectra}, \cite{moduli_problems_for_structured_ring_spectra}. The latter was classically used to obtain obstructions to realizing a given $E_{*}E$-comodule algebra as a homology of a commutative ring spectrum, but we will only make use of the linear variant, with no multiplicative structures involved. 

In more detail, one constructs a tower of $\infty$-categories

\begin{center}
$\spectra_{E} \rightarrow \ldots \rightarrow \mathcal{M}_{1}^{top} \rightarrow \mathcal{M}_{0}^{top}$,
\end{center}
where $\spectra_{E} \simeq \varprojlim \mathcal{M}_{l}^{top}$, $\mathcal{M}^{top}_{0} \simeq \ComodE$ and the functor $\spectra_{E} \rightarrow \mathcal{M}^{top}_{0}$ can be identified with the $E_{*}$-homology functor. Then, the results of Goerss and Hopkins imply there are obstructions to lifting objects, morphisms and homotopies between these along the functors $\mathcal{M}^{top}_{l+1} \rightarrow \mathcal{M}^{top}_{l}$, and that these obstructions live in $\Ext$-groups computed in $\ComodE$, the abelian category of $E_{*}E$-comodules. 

We then observe the classical fact that if $p > n+1$, then the category $\ComodE$ has finite homological dimension equal to $n^{2}+n$, so that the $\Ext$-groups above that degree vanish. This, coupled with Goerss-Hopkins theory, implies that the functors $\spectra_{E} \rightarrow \mathcal{M}_{l}^{top}$ induce an equivalence of homotopy $k$-categories $h_{k}\spectra_{E}  \simeq h_{k}\mathcal{M}_{l}^{top}$ for $l \geq k+n^{2}+n-1$, as in this range all the possible obstructions to lifting objects, morphisms and homotopies between morphisms vanish. Then, $\mathcal{M}_{l}^{top}$ for an appropriately chosen $l$ will be the $\infty$-category denoted above by $\spectra_{E}^{\Diamond}$, the "simplified version" of the $E$-local homotopy theory. 

We also construct an algebraic Goerss-Hopkins tower of the form

\begin{center}
$\dcat(E_{*}E) \rightarrow \ldots \rightarrow \mathcal{M}_{1}^{alg} \rightarrow \mathcal{M}_{0}^{alg}$;
\end{center}
it has the same formal properties. Namely, $\dcat(E_{*}E) \simeq \varprojlim \mathcal{M}_{l}^{alg}$, $\mathcal{M}_{0}^{alg} \simeq \ComodE$ and the functors $\mathcal{M}_{l+1}^{alg} \rightarrow \mathcal{M}^{alg}_{l}$ admit obstructions to lifting objects, morphisms and homotopies between these. The existence of this tower is intuitively why $\dcat(E_{*}E)$ is the appropriate "algebraic mirror" to the $E$-local category, an insight due to Franke. 

If $p > n+1$, then the same argument implies that the functors $\dcat(E_{*}E) \rightarrow \mathcal{M}_{l}^{alg}$ induce equivalences $h_{k}\dcat(E_{*}E) \simeq h_{k}\mathcal{M}_{l}^{alg}$ between homotopy $k$ categories for $l \geq k+n^{2}+n-1$. Then, $\mathcal{M}_{l}^{alg}$ for an appropriately chosen $l$ is the needed "simplified version" of the derived $\infty$-category. This completes the steps $(1)$ and $(2)$ in the outline of the proof sketched above. 

In step $(3)$, the construction of the equivalence $\mathcal{M}^{top}_{l} \simeq \mathcal{M}^{alg}_{l}$, we follow the ideas of Bousfield from his paper on an algebraic model for $(KU)_{p}$-local spectra \cite{bousfield1985homotopy}. 

By results of Hovey and Strickland, the stable $\infty$-category $\spectra_{E}$ and the abelian category $\ComodE$ do not depend on the choice of $E$, but only on the prime and the height. Here, it is technically convenient to work with one such that $E_{*}$ is concentrated in degrees divisible by $(2p-2)$, such as the Johnson-Wilson theory $E(n)$. 

We say that an $E_{*}E$-comodule $M$ is \emph{pure of phase $\varphi$} if it is concentrated in degrees $d =  \varphi$ modulo $(2p-2)$, where $\varphi \in \mathbb{Z}/(2p-2)$. By extension, we say that a spectrum is \emph{pure of phase $\varphi$} when its homology is.

As a consequence of sparsity combined with the finite homological dimension of $\ComodE$, we see that when $2p -2 > n^{2}+n$, the Adams spectral sequence $\Ext^{t, s}_{E_{*}E}(E_{*}X, E_{*}Y) \rightarrow [X,Y]^{t-s}$ collapses and induces an isomorphism $\Hom_{E_{*}E}(E_{*}X, E_{*}Y) \simeq [X, Y]$ for any two pure $E$-local spectra of the same phase. Moreover, Goerss-Hopkins theory implies that any pure comodule can be realized by homology of a spectrum. 

Taken together, the above two observations imply that at large primes the homology functor $E_{*}: \spectra_{E} \rightarrow \ComodE$ induces an equivalence $h_{t} \spectra_{E}^{\varphi} \simeq \ComodE^{\varphi}$ between the homotopy $t$-category of pure $E$-local spectra and the category of pure comodules of fixed phase $\varphi$. A calculation with the Adams-Novikov spectral sequence, which collapses in this case, shows that $t = 2p-2-n^{2}-n$. 

Since we assume that $E_{*}E$ is concentrated in degrees divisible by $(2p-2)$, any comodule is a sum of pure ones. This allows us to define a functor $\beta: \ComodE \rightarrow h_{t} \spectra_{E}$ given by 

\begin{center}
$\beta(\bigoplus M^{\varphi}) = \bigvee R^{\varphi}(M^{\varphi})$,
\end{center}
where the sum is taken over phases $\varphi \in \mathbb{Z}/(2p-2)$ and $R^{\varphi}(M^{\varphi})$ is the unique $E$-local spectrum with homology isomorphic to $M^{\varphi}$, which by the above discussion is well-defined as an object of $h_{t}\spectra_{E}$. 

By construction, we have $E_{*}(\beta(M)) \simeq M$ for any comodule $M$, and so we call $\beta$ the \emph{Bousfield splitting}. The spectra in the essential image are exactly those which are \emph{split} in the sense that they can be written as a finite sum of pure spectra. 

Intuitively, the Bousfield splitting $\beta: \ComodE \rightarrow h_{t}\spectra_{E}$ codifies not only that at large primes all comodules are realizable as a homology of a split spectrum, but also that such a realization is weakly functorial in the comodule. Here, we say weakly since we do not have a functor $\ComodE \rightarrow \spectra_{E}$, but rather only a functor into an appropriate homotopy category. At the same time, since $t$ grows with $p$, the functoriality gets stronger as the prime grows.

We then show that the Bousfield splitting induces a functor $\beta_{*}: \mathcal{M}_{l}^{alg} \rightarrow \mathcal{M}_{l}^{top}$ between the layers of Goerss-Hopkins towers for $l < t$. The functor $\beta_{*}$ is defined as an appropriate Kan extension using the explicit description of $\mathcal{M}^{top}_{l}$ coming from the work of the author on synthetic spectra \cite{pstrkagowski2018synthetic}. Informally, the reason that the weak functoriality turns out to be enough is that $\mathcal{M}_{l}^{top}$ is an $(l+1)$-category and so for $l < t$ it doesn't detect the difference between $\spectra_{E}$ and its homotopy $t$-category. 

Finally, we verify that $\beta_{*}: \mathcal{M}_{l}^{alg} \rightarrow \mathcal{M}_{l}^{top}$ is an equivalence. Choosing the prime large enough so that $l$ is in a range where $h\spectra_{E} \simeq h\mathcal{M}^{top}_{l}$ and $h\dcat(E_{*}E) \simeq h\mathcal{M}^{alg}_{l}$ finishes step $(3)$ and gives a proof of \textbf{Theorem \ref{thm:intro_homotopy_kcategories_of_elocal_spectra_and_periodic_chain_complexes_equivalent_at_large_primes}}. 

\subsection{Notation and conventions}

We will use $p$ to denote the prime and $n$ to denote the height. These will often be fixed, in which case we will let $E$ denote a $p$-local Landweber exact homology theory of height $n$. We will sometimes use the shorthand $q = 2p-2$. 

We use the theory of $\infty$-categories, as developed by Lurie and Joyal \cite{lurie_higher_topos_theory}. Unless explicitly stated otherwise, all constructions, in particular limits and colimits, should be understood in this sense. 

We say an $\infty$-category $\ccat$ is a \emph{$k$-category} if the mapping space $\map_{\ccat}(c, c^\prime)$ is $(k-1)$-truncated for any $c, c^\prime \in \ccat$. If $\ccat$ is an $\infty$-category, then there exists a $k$-category $h_{k}\ccat$ equipped with a functor $\ccat \rightarrow h_{k}\ccat$ which is universal with respect to this property, which we call the \emph{homotopy $k$-category.} \cite{lurie_higher_topos_theory}[2.3.4.12]. 

Using the explicit construction of Lurie, one can take $h_{k}\ccat$ to have the same objects as $\ccat$. The mapping spaces are given by $\map_{h_{k}\ccat}(c, c^\prime) \simeq \map_{\ccat}(c, c^\prime)_{\leq k-1}$. If $k = 1$, this is just the usual homotopy category, which we denote by $h\ccat$. 

\subsection{Acknowledgements}

I would like to thank my supervisor, Paul Goerss, for his support and wisdom, as well as for suggesting this problem. I would like to thank Tobias Barthel for the useful conversations we had on this subject. 

\section{Realization of comodules}
\label{section:realization_of_comodules}

In this section we show that if $E$ is a $p$-local Landweber exact homology theory of height $n$ and $2p-2 > n^2+n+2$, then any $E_{*}E$-comodule can be realized as a homology of a spectrum of a particular form, which we call split. We then show that the formation of a split realization is weakly functorial in the comodule. 

We start by introducing the notion of a height of Landweber exact homology theory. By $BP$ we denote the Brown-Peterson spectrum with $BP_{*} \simeq \mathbb{Z}_{(p)}[v_{1}, v_{2}, \ldots]$ and by $I_{n}$ the invariant ideals $I_{n} = (p, v_{0}, \ldots, v_{n-1})$. Let $E$ be a $p$-local multiplicative Landweber exact homology theory, so that $E_{*}$ is a $BP_{*}$-algebra.

\begin{defin}[\cite{hovey2003comodules}(4.1)]
We say $E$ is \emph{height $n$} if $E_{*} / I_{n} \neq 0$, but $E_{*} / I_{n+1} = 0$. 
\end{defin}
Let us give a few examples.

\begin{example}
\label{example:johnson_wilson_theory}
The Johnson-Wilson theory spectrum $E(n)$ with $E(n)_{*} \simeq \mathbb{Z}_{(p)}[v_{1}, \ldots, v_{n-1}, v_{n}^{\pm1}]$ is Landweber exact of height $n$. 
\end{example}

\begin{example}
To any perfect field $k$ of characteristic $p$ and a formal group law $\Gamma$ over $k$ of height $n$ one can associate the \emph{Lubin-Tate theory} spectrum $E(k, \Gamma)$, also known as the \emph{Morava $E$-theory}. We have $E(k, \Gamma)_{*} \simeq W(k)[[u_{1}, \ldots, u_{n-1}]][u^{\pm1}]$ and one verifies easily that $E(k, \Gamma)$ is Landweber exact of height $n$. 
\end{example}
If $E$ is a Landweber exact homology theory, then it is Adams \cite{hovey2003homotopy}[1.4.9], so that $E_{*}E$ has a canonical structure of a Hopf algebroid and $E_{*}X$ carries a structure of an $E_{*}E$-comodule for any spectrum $X$. The resulting category $\ComodE$ of comodules is symmetric monoidal Grothendieck abelian, with the tensor product induced from that of $E_{*}$-modules. 

It is a deep result of Hovey and Strickland that for any two Landweber exact homology theories of the same height the associated categories $\ComodE$ of $E_{*}E$-comodules are equivalent as symmetric monoidal categories, since they are both localizations of $\mathcal{C}omod_{BP_{*}BP}$ at the same class of maps \cite{hovey2003comodules}[4.4]. It follows that for our purposes all such homology theories are essentially interchangable.

The following fact is a basis of our approach to the $E$-local category at large primes, it appears to be folklore and can be found in the work of Franke \cite{franke1996uniqueness}[3.4.3.9], but perhaps not in a language easily understood by algebraic topologists. Thus, for the convenience of the reader we recall the argument here.

\begin{thm}
\label{thm:injective_dimension_of_ee_comodules}
Let $E$ be Landweber exact homology theory of height $n$ and assume that $p > n+1$. Then, the category $\ComodE$ of $E_{*}E$-comodules has homological dimension $n^{2}+n$, that is, for any comodules $M, N$ we have $\Ext_{E_{*}E}^{s, t}(M, N) = 0$ for $s > n^{2}+n$. 
\end{thm}

\begin{proof}
Using the chromatic spectral sequence, one shows that $\Ext_{E_{*}E}^{s, t}(E_{*}, E_{*}) = 0$ for $s > n^{2} + n$ if $p > n+1$, as a consequence of Morava vanishing. We do not reproduce the proof of that fact here, as it can be found in \cite{hovey1999invertible}[5.1].

Let us say that a comodule $N$ is a \emph{good target} if $\Ext_{E_{*}E}^{s, t}(E_{*}, N) = 0$ for $s > n^{2}+n$; by what was said above we know that the shifts of $E_{*}$ are good targets. Similarly, the long exact sequence of $\Ext$-groups implies that if $0 \rightarrow N \rightarrow N^{\prime} \rightarrow N^{\prime \prime} \rightarrow 0$ is short exact and $N$ is a good target, then $N^{\prime}$ is a good target if and only if $N^{\prime \prime}$ is. We then deduce by induction on the length of the Landweber filtration in the sense of \cite{hovey2003comodules} that all finitely generated comodules are good targets. 

Now, any $E_{*}E$-comodule is a filtered colimit of finitely generated ones. Since $E_{*}$ is dualizable, we know that $\Ext^{s, t}(E_{*}, N)$ can be computed using the cobar complex and so commutes with filtered colimits in $N$. Thus, we deduce that all comodules are good targets.

Let us now say that $N$ is a \emph{good source} if $\Ext^{s, t}_{E_{*}E}(N, M) = 0$ for $s > n^{2} + n$ and all $M$, by the above $E_{*}$ is good source. The condition of being a good source is clearly closed under direct sums, and under taking quotients, as we see using the long exact sequence of $\Ext$-groups. Since $E_{*}$ generates the whole category of $E_{*}E$-comodules by \cite{hovey2003comodules}[5.1], it follows that every comodule is a good source, which is what we wanted to show.
\end{proof}

\begin{rem}
\label{rem:comode_generated_by_objects_of_finite_homological_dimension}
If $p \leq n+1$, then the category $\ComodE$ is not of finite homological dimension. However, the argument given in the proof of \textbf{Theorem \ref{thm:injective_dimension_of_ee_comodules}} is enough to establish that $\ComodE$ is always generated under colimits by objects $M$ of injective dimension $n^{2}+n$, that is, such that $\Ext^{t, s}(M, N) = 0$ for $s > n^{2}+n$ and all $N$. In fact, one can take $M$ to be $E_{*}(DP)$, where $DP$ is the Whitehead dual of the Hopkins-Ravenel finite spectrum of \cite{ravenel2016nilpotence}[8.3].
\end{rem}

\begin{rem}
\label{rem:dcomode_is_left_complete}
Since at any prime and height $\ComodE$ is generated by objects of finite homological dimension, see \textbf{Remark \ref{rem:comode_generated_by_objects_of_finite_homological_dimension}}, it follows that the derived $\infty$-category $\dcat(\ComodE)$ is always left complete.
\end{rem}

For technical reasons, it will be convenient for us to work with a Landweber exact homology theory $E$ of height $n$ such that $E_{*}$ is concentrated in degrees divisible by $q = (2p-2)$. For example, as one can take Johnson-Wilson theory of \textbf{Example \ref{example:johnson_wilson_theory}}. Notice that when $E_{*}$ is concentrated in degrees divisible by $q$, then so is $E_{*}E \simeq E_{*} \otimes _{BP_{*}} BP_{*}BP \otimes _{BP_{*}} E_{*}$. 

\begin{defin}
If $\varphi \in \mathbb{Z}/q$, then we say that an $E_{*}E$-comodule $M$ is \emph{pure of phase $\varphi$} if it is concentrated in degrees $d = \varphi$ modulo $q$. We denote the category of pure comodules of phase $\varphi$ by $\ComodE^{\varphi}$.
\end{defin}

We have an equivalence of categories $\prod _{\varphi \in \mathbb{Z}/q} \ComodE^{\varphi} \simeq \ComodE$, where the first term is the product of the categories of pure comodules of different phases. An explicit formula for this equivalence is given by $(M^{\varphi})_{\varphi \in \mathbb{Z}/q} \mapsto \bigoplus _{\varphi \in \mathbb{Z}/q} M^{\varphi}$; in other words, any comodule is a finite sum of pure ones, and canonically so. 

\begin{rem}
Note that even if $E_{*}$ isn't concentrated in degrees divisible by $q$, the category $\ComodE$ of comodules still enjoys a splitting as a $q$-fold product, since by results of Hovey and Strickland, $\ComodE$ only depends on the height. 

In this case, the splitting cannot be defined in terms of degree alone. For example, in his work at $n =1$, Bousfield uses a splitting in terms of the eigenvalues of Adams operations \cite{bousfield1985homotopy}.
\end{rem}

\begin{defin}
\label{defin:pure_and_split_spectra}
We say a spectrum $X$ is \emph{pure of phase $\varphi$} if its homology $E_{*}X$ is. We say $X$ is \emph{split} if it is a finite sum of pure spectra. 
\end{defin}
In the $n = 1$ case, the notion of a split spectrum appears in the work of Bousfield under the name of a \emph{generalized Eilenberg-MacLane spectrum} \cite{bousfield1985homotopy}. Notice that unlike in the case of comodules, where every comodule can be written as a finite sum of its pure parts, not every spectrum is split.

\begin{thm}
\label{thm:at_large_primes_every_ee_comodule_is_realizable}
Let $2p > n^2+n$ and $p > n+1$. Then any $E_{*}E$-comodule $M$ is realizable by a homology of a split spectrum; that is, there exists a split spectrum $X$ such that $E_{*}X \simeq M$.
\end{thm}

\begin{proof}
Since any comodule is a finite sum of pure ones, we can assume that $M$ is pure. Then, by Goerss-Hopkins obstruction theory, there are inductively defined obstructions to realizing $M$ as a homology of a spectrum lying in the groups $\Ext_{E_{*}E}^{k+2, k}(M, M)$ for $k \geq 1$ \cite{moduli_spaces_of_commutative_ring_spectra}. These groups vanish by purity of $M$ when $k < 2p-2$, since there's nothing in these internal degrees, and by \textbf{Theorem \ref{thm:injective_dimension_of_ee_comodules}}, when $k \geq 2p-2$, since then $k+2 \geq 2p > n^2+n$.
\end{proof}

\begin{rem}
The result of \textbf{Theorem \ref{thm:at_large_primes_every_ee_comodule_is_realizable}} appears to be folklore, although we do not know of a reference that proves it in this generality. Hovey and Strickland observe that this follows from Franke's work \cite{hovey1999morava}, but a gap in the latter has been found by Patchkoria \cite{patchkoria2017exotic}.

An important test case is the existence of $E$-local Toda-Smith complexes, that is, of spectra $V(i)$ such that $E_{*} V(i) \simeq E_{*} / I_{i}$. In this case, the observation that $V(i)$ can be constructed inductively at large primes appears in the work of Devinatz \cite{devinatz2008towards}.
\end{rem}
The rest of the section will be devoted to a strengthening of \textbf{Theorem \ref{thm:at_large_primes_every_ee_comodule_is_realizable}}, where we show that a split realization of a comodule is not only unique, but can also be made weakly functorial in the underlying comodule. Our method here will be the Adams spectral sequence. 

Recall that if $X, Y$ are $E$-local, then to show that the $E_{2}$-term of the of the Adams spectral sequence converging to $[X, Y]$ can be identified with $\Ext$-groups in comodules one identifies the former with the homology of the cobar complex. However, the cobar complex only computes the $\Ext$-groups under the assumption that $E_{*}X$ is projective, as it is not a resolution by injectives, but only by relative injectives \cite{ravenel_complex_cobordism}[App. 1]. 

This issue can be avoided by working with the so-called \emph{modified Adams spectral sequence}. To construct it, one observes that if $I$ is an $E$-local spectrum such that $E_{*}I$ is an injective comodule, then $[X, I] \simeq \Hom_{E_{*}E}(E_{*}X, E_{*}I)$ as a consequence of Brown representability. The modified Adams spectral sequence is obtained by resolving $Y$ by such injective spectra and applying $[X, -]$; for the details, see \cite{devinatz1997morava}. 

It is clear from the construction that the $E_{2}$-term coincides with $\Ext$-groups. Moreover, we will only be working with this spectral sequence when $p > n+1$, in which case \textbf{Theorem \ref{thm:injective_dimension_of_ee_comodules}} guarantees that the resolution of $Y$ can be chosen to be finite, so that the spectral sequence clearly converges in a finite number of steps. 

\begin{lemma}
\label{lemma:homology_functor_highly_connective_on_mapping_spaces}
Suppose that $2p-2 > n^{2}+n$ and let $X, Y$ be pure $E$-local spectra of the same phase. Then, $\map_{\spectra_{E}}(X, Y) \rightarrow \Hom_{E_{*}E}(E_{*}X, E_{*}Y)$ is a $(2p-2-n^2-n)$-connected map of spaces. 
\end{lemma}

\begin{proof}
Notice that since $2p-2 > n^{2}+n$, we have $p > n+1$. It follows that we have a strongly convergent modified Adams spectral sequence 

\begin{center}
$\Ext^{s, t}_{E_{*}E}(E_{*}X, E_{*}Y) \Rightarrow [X, Y]^{t-s}$.
\end{center}
By \textbf{Theorem \ref{thm:injective_dimension_of_ee_comodules}}, the $E_{2}$-term has a horizontal vanishing line at $n^{2}+n$ and by assumption of purity, it is concentrated in internal degrees divisible by $2p-2$. It follows that the spectral sequence collapses on the second page. 

Then, notice that the only non-zero group on the $t - s = 0$ line is $\Ext^{0, 0}_{E_{*}E}(E_{*}X, E_{*}Y)$ and that all groups vanish for $0 < t-s < 2p - 2 - n^2 - n$. This ends the argument. 
\end{proof}

\begin{thm}
\label{thm:homology_on_pure_spectra_gives_an_equivalence_with_homotopy_category}
Let $\varphi \in \mathbb{Z}/q$ and let $\spectra_{E}^{\varphi}$ be the $\infty$-category of pure $E$-local spectra of phase $\varphi$. Then, the functor $E_{*}: \spectra_{E} \rightarrow \ComodE$ induces an equivalence $h_{k}\spectra_{E}^{\varphi} \simeq \ComodE^{\varphi}$ between the homotopy $k$-category of $\spectra_{E}^{\varphi}$, where $k = 2p-2-n^{2}-n$, and the category of pure comodules of phase $\varphi$. 
\end{thm}

\begin{proof}
The statement of the theorem is equivalent to saying that $E_{*}: \spectra_{E}^{\varphi} \rightarrow \ComodE^{\varphi}$ induces an equivalence between mapping spaces after applying $(k-1)$-truncation, and that it is essentially surjective. This is exactly \textbf{Lemma \ref{lemma:homology_functor_highly_connective_on_mapping_spaces}} and \textbf{Theorem \ref{thm:at_large_primes_every_ee_comodule_is_realizable}}. 
\end{proof}
Notice that \textbf{Theorem \ref{thm:homology_on_pure_spectra_gives_an_equivalence_with_homotopy_category}} can be interpreted as saying that the unique realization of a pure comodule is weakly functorial. Here, by saying \emph{weakly functorial} we mean that it is strictly functorial if we consider the realization as an object of the homotopy $k$-category, rather than the $\infty$-category $\spectra_{E}$ itself. However, as $k$ grows together with $p$, the functoriality gets stronger the larger the prime. 

We will now extend this functoriality of \textbf{Theorem \ref{thm:homology_on_pure_spectra_gives_an_equivalence_with_homotopy_category}} to all comodules by allowing finite sums, that is, by working with split spectra. 

\begin{defin}
\label{defin:bousfield_splitting}
For each $\varphi \in \mathbb{Z}/q$, let $R^{\varphi}: \ComodE^{\varphi} \rightarrow h_{k}\spectra_{E}^{\varphi}$ denote the inverse to the equivalence of \textbf{Theorem \ref{thm:homology_on_pure_spectra_gives_an_equivalence_with_homotopy_category}}. Then, the \emph{Bousfield splitting functor} $\beta: \ComodE \rightarrow h_{k}\spectra_{E}$ is defined by 

\begin{center}
$\beta(M) = \bigvee_{\varphi \in \mathbb{Z}/q} R^{\varphi}(M^{\phi})$,
\end{center}
where $M = \bigoplus _{\varphi \in \mathbb{Z}/q} M^{\varphi}$ is the pure decomposition of $M$. 
\end{defin}

\begin{rem}
\label{rem:bousfield_splitting_a_splitting}
By construction we have $E_{*}(\beta M) \simeq M$, justifying the name. Note that the spectra in the image of $\beta$ are exactly the split $E$-local spectra.
\end{rem}

We will now prove that when restricted to comodules projective over the base ring $E_{*}$, the Bousfield splitting $\beta$ can be given a structure of a symmetric monoidal functor. For this purpose, it will be convenient for us to identify the abelian group $\mathbb{Z}/q$ with a discrete symmetric monoidal category. 

\begin{lemma}
If $\ccat$ is presentably symmetric monoidal, then the functor $\infty$-category $\Fun(\mathbb{Z}/q, \ccat)$ carries the symmetric monoidal structure of Day convolution, informally given by 

\begin{center}
$(X \otimes Y)(\varphi) = \bigoplus _{\kappa + \lambda = \varphi} X(\kappa) \otimes Y(\lambda)$,
\end{center}
where $\varphi, \kappa, \lambda \in \mathbb{Z}/q$. 
\end{lemma}

\begin{proof}
Since $\ccat$ is presentably symmetric monoidal, it can be identified with a commutative monoid in the $\infty$-category $Pr^{L}$ of presentable $\infty$-categories and cocontinuous functors \cite{higher_algebra}[4.8]. Then, we have $\Fun(\mathbb{Z}/q, \ccat) \simeq P(\mathbb{Z}/q^{op}) \otimes \ccat$, where the tensor product is taken in $Pr^{L}$ and by $P(\mathbb{Z}/q^{op})$ we denote the $\infty$-category of presheaves of spaces on $\mathbb{Z}/q^{op}$. 

The $\infty$-category $P(\mathbb{Z}/q^{op})$ carries a unique presentable symmetric monoidal structure induced from the one of $\mathbb{Z}/q^{op}$, see \cite{higher_algebra}[4.8.1.12], and hence so does $\Fun(\mathbb{Z}/q, \ccat) \simeq P(\mathbb{Z}/q^{op}) \otimes \ccat$. One can verify this symmetric monoidal structure is given by the informal formula above. 
\end{proof}

Note that the projection from $\mathbb{Z}/q$ onto the trivial group induces a symmetric monoidal functor $\pi_{\ccat}^{\mathbb{Z}/q}: \Fun(\mathbb{Z}/q, \ccat) \rightarrow \ccat$, which is informally given by $X \mapsto \bigoplus _{\varphi \in \mathbb{Z}/q} X(\varphi)$. 

\begin{defin}
\label{defin:pure_diagram_of_comodules_or_spectra}
We say a functor $X: \mathbb{Z}/q \rightarrow \ComodE$ is a \emph{pure diagram} if $X(\varphi)$ is a pure comodule of phase $\varphi$ for all $\varphi \in \mathbb{Z}/q$. We say $Y: \mathbb{Z}/q \rightarrow \spectra_{E}$ is a \emph{pure diagram} if $E_{*}Y$ is. 
\end{defin}
In other words, $X: \mathbb{Z}/q \rightarrow \ComodE$ is a pure diagram when $X(\varphi)$ is concentrated in degrees $d = \varphi$ modulo $q$. Note that the subcategory $\Fun^{pure}(\mathbb{Z}/q, \ComodE)$ of pure diagrams is easily seen to contain the unit and be stable under Day convolution, and so its inherits a symmetric monoidal structure from $\Fun(\mathbb{Z}/q, \ComodE)$.

\begin{lemma}
\label{lemma:pure_diagrams_in_comodules_the_same_as_comodules}
The restriction $\pi_{\ComodE}^{\mathbb{Z}/q}: \Fun^{pure}(\mathbb{Z}/q, \ComodE) \rightarrow \ComodE$ of $\pi_{\ComodE}^{\mathbb{Z}/q}$ to the subcategory of pure diagrams is a symmetric monoidal equivalence.  
\end{lemma}

\begin{proof}
Since $\pi_{\ComodE}^{\mathbb{Z}/q}$ is given by the fomula $X \mapsto \bigoplus _{\varphi \in \mathbb{Z}/q} X(\varphi)$, to say that it is an equivalence is the same as saying that any comodule is uniquely a sum of pure ones, which is clear. 
\end{proof}

\begin{thm}
\label{thm:bousfield_splitting_is_symmetric_monoidal}
The restriction $\beta: \ComodE^{proj} \rightarrow h_{k}\spectra_{E}$ of the Bousfield splitting to the category of $E_{*}$-projective comodules carries a symmetric monoidal structure.
\end{thm}

\begin{proof}
First, observe that by the K\"unneth formula the restriction $E_{*}: \spectra_{E}^{proj} \rightarrow \ComodE^{proj}$ of the homology functor to spectra with projective homology is symmetric monoidal. By \textbf{Theorem \ref{thm:homology_on_pure_spectra_gives_an_equivalence_with_homotopy_category}}, applied to each phase $\varphi \in \mathbb{Z}/q$ separately, the symmetric monoidal functor

\begin{center}
$E_{*}: \Fun^{pure}(\mathbb{Z}/q, \spectra_{E}^{proj}) \rightarrow \Fun^{pure}(\mathbb{Z}/q, \ComodE^{proj})$
\end{center}
between the $\infty$-categories of pure functors in the sense of \textbf{Definition \ref{defin:pure_diagram_of_comodules_or_spectra}} identifies its target with the homotopy $k$-category of the source. It follows that we have a chain of symmetric monoidal equivalences

\begin{center}
$h_{k} \Fun^{pure}(\mathbb{Z}/q, \spectra_{E}^{proj}) \simeq \Fun^{pure}(\mathbb{Z}/q, \ComodE^{proj}) \simeq \ComodE^{proj}$,
\end{center}
where the second one is \textbf{Lemma \ref{lemma:pure_diagrams_in_comodules_the_same_as_comodules}}. Thus, the composite defines a symmetric monoidal equivalence $\alpha: \ComodE^{proj} \rightarrow h_{k}\Fun^{pure}(\mathbb{Z}/q, \spectra_{E}^{proj})$. 

Then, one verifies easily that the Bousfield splitting $\beta$ can be identified with $\pi_{\spectra_{E}}^{\mathbb{Z}/q} \circ \alpha$, where $\pi_{\spectra_{E}}^{\mathbb{Z}/q}: h_{k} \Fun^{pure}(\mathbb{Z}/q, \spectra_{E}^{proj}) \rightarrow h_{k}\spectra_{E}^{proj}$ is induced by the projection from $\mathbb{Z}/q$ to the trivial group. It follows that $\beta$ is symmetric monoidal, being a composite of symmetric monoidal functors. 
\end{proof}

\section{Derived $\infty$-category of differential comodules}
\label{section:derived_inftycat_of_differential_comodules}
In this section we introduce the algebraic analogue of the $E$-local category, what we will call the \emph{derived $\infty$-category of $E_{*}E$}, also referred to in the literature as the \emph{periodic derived $\infty$-category}.

We will only need the constructions that appear here for the Hopf algebroid $E_{*}E$ associated to a Landweber exact homology theory, so that's the notation we use, but everything we say here holds for an arbitrary Hopf algebroid. 

\begin{defin}
A \emph{differential $E_{*}E$-comodule} is a pair $(M, d)$, where $M \in \ComodE$ and $d: M \rightarrow M$ is an endomorphism of degree $1$ satisfying $d^{2} = 0$. We denote the abelian category of differential comodules by $d\ComodE$. 
\end{defin}
We would like to equip the category of differential comodules with a model structure. To do so, it is convenient to think of a differential comodule as defining a whole chain complex of comodules, one where all the terms are given by shifts of $M$. 
\begin{defin}
A \emph{periodic chain complex} is a pair $(C^{\bullet}, \phi)$, where $C_{\bullet} \in \mathcal{C}h(\ComodE)$  is a chain complex and $\phi_{\bullet}: C_{\bullet-1} \rightarrow C_{\bullet}[1]$ is an isomorphism between the internal and external shifts. We denote the abelian category of periodic chain complexes by $\mathcal{C}h^{per}(\ComodE)$. 
\end{defin}
The notion of a periodic chain complex is in fact equivalent to that of a differential comodule, as we now show. 

\begin{prop}
\label{prop:differential_comodules_and_periodic_chain_complexes_are_equivalent}
The functor $f: \mathcal{C}h^{per}(\ComodE) \rightarrow d\ComodE$ defined by 

\begin{center}
$f(C_{\bullet}, \phi_{\bullet}) =  (C_{0}, \phi_{0} \circ d_{0}: C_{0} \rightarrow C_{0}[1])$ 
\end{center}
is an equivalence of categories. 
\end{prop}

\begin{proof}
One can verify that $g: d\ComodE \rightarrow \mathcal{C}h^{per}(\ComodE)$ defined by $g(M, d)_{k} = M[-k]$, with the differential given by shifts of $d$, is an explicit inverse equivalence to $f$. 
\end{proof}
We have decided to work with differential objects, rather than periodic chain complexes, since they are more natural and are better established in algebra. On the other hand, the literature on exotic equivalences is largely written in terms of periodic chain complexes, starting with the work of Franke \cite{franke1996uniqueness}. 

Barnes and Roitzheim show that the category $\mathcal{C}h^{per}(\ComodE)$ can be identified with the category of modules in chain complexes over the \emph{periodicized unit} $P(\monunit)$ \cite{barnes2011monoidality}. Using this identification one equips the category of periodic chain complexes with a tensor product and shows that it inherits a compatible model structure from $\mathcal{C}h(\ComodE)$.

In our case, there is an isomorphism $P(\monunit) \simeq E_{*}[\tau^{\pm 1}]$, where $| \tau | = (1, -1)$ and we equip the latter with the zero differential. To see this, notice that to give a structure of a module over $E_{*}[\tau^{\pm 1}]$ it is enough to specify the action of $\tau$, which can be any isomorphism of degree $(1, -1)$, so that datum of a module is equivalent to one of a periodic chain complex. 

\begin{rem}
The paper \cite{barnes2011monoidality} contains a subtle mistake, since the periodicized unit $P(\monunit)$ is not always commutative in the generality stated there. It is, however, commutative in the present case, as we see from the above description, since $\tau$ is in even total degree. 
\end{rem}

Through \textbf{Proposition \ref{prop:differential_comodules_and_periodic_chain_complexes_are_equivalent}}, one can transfer this symmetric monoidal model structure to differential comodules. Explicitly, a map $(M, d_{M}) \rightarrow (N, d_{N})$ of differential comodules is a weak equivalence if and only if it an isomorphism, that is, when it induces an isomorphism $ker(d_{M}) / im(d_{M}) \simeq ker(d_{N}) / im(d_{N})$, and the tensor product lifts the one on comodules. 

The monoidality of the structure ensures that the tensor product preserves weak equivalences between cofibrant objects, allowing us to localize category of the latter and obtain the underlying symmetric monoidal $\infty$-category in the sense of \cite{higher_algebra}[4.1.3.6]. 

\begin{defin}
\label{defin:derived_category_of_ee}
We call the underlying symmetric monoidal $\infty$-category of $d\ComodE$ the \emph{derived $\infty$-category of $E_{*}E$} and denote it by $\dcat(E_{*}E)$. 
\end{defin}

Note that by \textbf{Proposition \ref{prop:differential_comodules_and_periodic_chain_complexes_are_equivalent}}, $\dcat(E_{*}E)$ is canonically equivalent to the \emph{periodic} derived $\infty$-category, which is the symmetric monoidal $\infty$-category underlying $\mathcal{C}h^{per}(\ComodE)$. It is the latter $\infty$-category that was considered by Barthel, Schlank and Stapleton in their proof of asymptotic algebraicity \cite{barthel2017chromatic}.

\begin{warning}
It should be stressed that $\dcat(E_{*}E)$ is an $\infty$-category distinct from $\dcat(\ComodE)$, the derived $\infty$-category of the abelian category $\ComodE$. Rather, the relation between the two is given by \textbf{Remark \ref{rem:derived_cat_of_differential_comodules_as_modules_in_the_derived_category}}.
\end{warning}

\begin{rem}
\label{rem:derived_cat_of_differential_comodules_as_modules_in_the_derived_category}
By \textbf{Proposition \ref{prop:differential_comodules_and_periodic_chain_complexes_are_equivalent}} and the work of Barnes and Roitzheim, $\dcat(E_{*}E)$ is equivalent to the underlying $\infty$-category of $P(\monunit)$-modules in chain complexes. Then, since $P(\monunit)$ is cofibrant, \cite{higher_algebra}[4.3.3.17] implies that we have a symmetric monoidal equivalence

\begin{center}
$\dcat(E_{*}E) \simeq \Mod_{P(\monunit)}(\dcat(\ComodE))$
\end{center}
between the derived $\infty$-category of differential comodules and modules over $P(\monunit)$ in the derived $\infty$-category of $\ComodE$. 
\end{rem}

\section{Goerss-Hopkins theory} 
\label{section:goerss_hopkins_theory}

In this section we construct the Goerss-Hopkins tower associated to a homology theory $E$, as well as its algebraic analogue \cite{moduli_spaces_of_commutative_ring_spectra}, \cite{moduli_problems_for_structured_ring_spectra}. We then show that when the prime is large enough, the layers of the towers stabilize in a precise sense. 

\begin{rem}
In this note we will only make use of \emph{linear} Goerss-Hopkins theory; that is, we will be working with synthetic spectra rather than with synthetic commutative ring spectra, and the obstructions will be valued in $\Ext$-groups. 

The obstructions alone could be alternatively obtained using the classical Toda obstruction theory \cite{margolis2011spectra}[16.3]. However, it is crucial to our approach that Goerss-Hopkins thoery yields not only obstructions, but well-behaved $\infty$-categories.
\end{rem}

By a \emph{synthetic spectrum} we mean a hypercomplete, connective synthetic spectrum based on $E$ in the sense of \cite{pstrkagowski2018synthetic}, and we denote their $\infty$-category by $\synspectra$. In other words, a synthetic spectrum is a hypercomplete, spherical sheaf of spaces on the site $\spectra_{E}^{fp}$ of finite, $E_{*}$-projective spectra. 

The $\infty$-category $\synspectra$ is Grothendieck prestable in the sense of \cite{lurie_spectral_algebraic_geometry}[C.1.4.2]; in other words, it is equivalent to the connective part of a presentable stable $\infty$-category equipped with a compatible $t$-structure. 

There is an equivalence $\synspectra^{\heartsuit} \simeq \ComodE$ between the heart of $\synspectra$, that is, the subcategory of discrete objects, and the abelian category of $E_{*}E$-comodules \cite{pstrkagowski2018synthetic}[4.16]. The identification of the heart defines for any synthetic spectrum $X$ a sequence of non-negatively graded homotopy groups valued in comodules, which we will denote by $\pi_{*}X$. These homotopy groups detect equivalences, since we work with hypercomplete synthetic spectra. 

\begin{rem}
Synthetic spectra can be considered as topological objects on their own right, in which case the notation $\pi_{*}X$ is better reserved for the geometric homotopy groups, that is, those obtained as homotopy classes of maps from the sphere, as was done in \cite{pstrkagowski2018synthetic}. 

In this note, we will not focus on geometric properties of $\synspectra$, instead using as much as possible only the abstract properties. In particular, $\pi_{*}X$ will always denote the $t$-structure homotopy groups valued in comodules.
\end{rem}

The $\infty$-category $\synspectra$ is graded symmetric monoidal, with the grading inducing the internal shift on comodule homotopy groups, and has a monoidal unit, which we will denote by $\monunit$. Its homotopy groups are given by the polynomial algebra $\pi_{*} \monunit \simeq E_{*}[\tau]$, where $| \tau | = (1, -1)$.

We have a tower $\monunit \rightarrow \ldots \rightarrow \monunitt{1} \rightarrow \monunitt{0}$ of commutative algebras given by Postnikov truncations of the unit, which in turn induces a tower of module $\infty$-categories 

\begin{center}
$\synspectra \rightarrow \ldots \rightarrow \Mod_{\monunitt{1}}(\synspectra) \rightarrow \Mod_{\monunitt{0}}(\synspectra)$;
\end{center}
Goerss-Hopkins theory arises from a detailed study of this tower. The identification of the heart $\synspectra^{\heartsuit} \simeq \ComodE$ extends to an equivalence $\Mod_{\monunitt{0}}(\synspectra) \simeq \dcat(\ComodE)_{\geq 0}$ \cite{pstrkagowski2018synthetic}[4.54], where the latter is the connective derived $\infty$-category of comodules. Thus,  the bottom of the tower always has a purely algebraic description, and so the tower can be interpreted as measuring the passage from algebra to topology. 

There's a functor $\nu: \spectra_{E} \rightarrow \synspectra$ called the \emph{synthetic analogue} and one can show that it is an embedding of $\infty$-categories  \cite{pstrkagowski2018synthetic}[4.37]. The essential image of $\nu$ is the $\infty$-category of those synthetic spectra $X$ such that $\monunitt{0} \otimes X$ is discrete, equivalently, such that $\pi_{*} X \simeq \pi_{*} \monunit \otimes _{\pi_{0} \monunit} \pi_{0} X$. This motivates the following definition, which is central to Goerss-Hopkins theory. 

\begin{defin}
\label{defin:topological_potential_l_stage}
A \emph{topological potential $l$-stage} is a $\monunitt{l}$-module $X$ in $\synspectra$ such that $\monunitt{0} \otimes _{\monunitt{l}} X$ is discrete. We denote the $\infty$-category of topological potential $l$-stages by $\mathcal{M}^{top}_{l}$
\end{defin}
The extension of scalars provides functors $u_{k}: \mathcal{M}_{l} \rightarrow \mathcal{M}_{k}$ for any $l \geq k$, these then assemble into a tower of $\infty$-categories, which we call the \emph{Goerss-Hopkins tower}. The main properties of this tower are as follows. 

\begin{prop}
\label{prop:properties_of_the_topological_gh_tower}
The tower $\mathcal{M}^{top}_{\infty} \rightarrow \ldots \rightarrow \mathcal{M}^{top}_{1} \rightarrow \mathcal{M}^{top}_{0}$ of $\infty$-categories of topological potential stages and extension of scalars functors has the following properties: 

\begin{enumerate}[label=($\spadesuit$\arabic*)]
\item $\mathcal{M}^{top}_{l}$ is an $(l+1)$-category
\item the functor $\pi_{0}: \mathcal{M}^{top}_{0} \rightarrow \ComodE$ is an equivalence
\item if $X \in \mathcal{M}^{top}_{l-1}$, then there exists an obstruction in $\Ext_{E_{*}E}^{l+2, l}(u_{0} X, u_{0} X)$ which vanishes if and only if $X$ can be lifted to a potential $l$-stage; that is, if there exists $\widetilde{X} \in \mathcal{M}^{top}_{l}$ such that $u_{l-1} \widetilde{X} \simeq X$
\item for any $X, Y \in \mathcal{M}_{l}^{top}$ with $l \geq 1$ there exists a fibre sequence 

\begin{center}
$\map_{\mathcal{M}^{top}_{l}}(X, Y) \rightarrow \map_{\mathcal{M}^{top}_{l-1}}(u_{l-1} X, u_{l-1}Y) \rightarrow \map_{\dcat(\ComodE)}(u_{0} X, \Sigma^{l+1} u_{0} Y [-l])$,
\end{center}

\item the functors $u_{l}: \mathcal{M}^{top}_{\infty} \rightarrow \mathcal{M}^{top}_{l}$ induce an equivalence $\mathcal{M}^{top}_{\infty} \simeq \varprojlim \mathcal{M}^{top}_{l}$

\item the synthetic analogue construction restricts to an equivalence $\nu: \spectra_{E} \rightarrow \mathcal{M}^{top}_{\infty}$ along which the functor $u_{0}: \mathcal{M}^{top}_{\infty} \rightarrow \mathcal{M}_{0}$ can be identified with $E_{*}: \spectra_{E} \rightarrow \ComodE$
\end{enumerate}
\end{prop}

\begin{proof}
This is \cite{moduli_problems_for_structured_ring_spectra}[3.3.2, 3.3.4, 3.3.5]. Note that the argument establishing $(\spadesuit 5)$, namely \cite{moduli_problems_for_structured_ring_spectra}[3.3.3], is incomplete, and a separate argument proving the Postnikov convergence of $\synspectra$ is needed. A complete account will follow in \cite{abstract_gh_theory}. 
\end{proof}
The algebraic analogue of the $E$-local category is given by derived $\infty$-category $\dcat(E_{*}E)$, the underlying $\infty$-category of the category $d\ComodE$ of differential comodules, which we have introduced in \S\ref{section:derived_inftycat_of_differential_comodules}. We will now construct an algebraic Goerss-Hopkins tower, one which has the same formal properties as the topological one, but has $\dcat(E_{*}E)$ as its limit. 

To construct the Goerss-Hopkins tower, we need an $\infty$-category playing the role of synthetic spectra, that is, a prestable $\infty$-category with a certain properties. By \textbf{Remark \ref{rem:derived_cat_of_differential_comodules_as_modules_in_the_derived_category}}, there is an equivalence $\dcat(E_{*}E) \simeq \Mod_{P(\monunit)}(\dcat(\ComodE))$, where the latter denotes $P(\monunit)$-modules in the derived $\infty$-category of the abelian category $\ComodE$. Here, $P(\monunit)$ is a certain commutative algebra with homotopy groups $\pi_{*} P(\monunit) \simeq E_{*}[\tau^{\pm 1}]$ constructed by Barnes and Roitzheim. 

\begin{defin}
\label{defin:periodicity_algebra}
We call the connective cover of $P(\monunit)$ the \emph{periodicity algebra} and denote it by $P := P(\monunit)_{\geq 0}$. It is a commutative algebra in the connective derived $\infty$-category $D(\ComodE)_{\geq 0}$ with homotopy groups $\pi_{*} P \simeq E_{*}[\tau]$. 
\end{defin}

\begin{rem}
\label{algebra_p_equivalent_to_the_free_algebra}
As an associative algebra in $\dcat(\ComodE)$, the periodicity algebra $P$ is equivalent to the free algebra on $\Sigma E_{*}[-1]$, the equivalence induced by the inclusion $\Sigma E_{*}[-1] \hookrightarrow P(\monunit)$ into the Barnes-Roitzheim algebra, which factors through $P$. 

However, this identification does not supply the commutative structure, which comes from the strictly commutative model in chain complexes of comodules. 
\end{rem}

Observe that the $\infty$-category $\Mod_{P}(\dcat(\ComodE)_{\geq 0})$ of connective $P$-modules is formally analogous to the $\infty$-category of synthetic spectra. That is, it is Grothendieck prestable, its heart is canonically equivalent to $\ComodE$, and it is symmetric monoidal with the homotopy groups of the unit given by $\pi_{*} P \simeq E_{*}[\tau]$. 

The Barnes-Roitzheim algebra $P(\monunit)$ can be recovered as a localization of $P$, and there is a fully faithful embedding 

\begin{center}
$\dcat(E_{*}E) \simeq \Mod_{P(\monunit)}(\dcat(\ComodE)) \hookrightarrow \Mod_{P}(\dcat(\ComodE)_{\geq 0})$
\end{center}
where the arrow is given by taking connective covers. One calculates with no difficulty that the essential image of that embedding is given by the $\infty$-category of those $P$-modules $M$ such that $P_{\leq 0} \otimes _{P} M$ is discrete, equivalently, that $\pi_{*} M \simeq \pi_{*} P \otimes _{\pi_{0} P} \pi_{0} M$. This suggests the following algebraic analogue of \textbf{Definition \ref{defin:topological_potential_l_stage}}. 

\begin{defin}
\label{defin:algebraic_potential_l_stage}
We say that a connective $P_{\leq l}$-module $M$ is an \emph{algebraic potential $l$-stage} if $P_{\leq 0} \otimes _{P_{\leq l}} M$ is discrete. We denote the $\infty$-category of algebraic potential $l$-stages by $\mathcal{M}^{alg}_{l}$.
\end{defin}
Clearly, extension of scalars then induces functors $u_{k}: \mathcal{M}_{l}^{alg} \rightarrow \mathcal{M}_{k}^{alg}$ for any $l \geq k$, leading to a tower of $\infty$-categories which we call the \emph{algebraic Goerss-Hopkins tower}. Its main properties are as follows. 

\begin{prop}
\label{prop:properties_of_the_algebraic_gh_tower}
The tower $\mathcal{M}^{alg}_{\infty} \rightarrow \ldots \rightarrow \mathcal{M}^{alg}_{1} \rightarrow \mathcal{M}^{alg}_{0}$ of $\infty$-categories of algebraic potential stages and extension of scalars functors has the following properties: 

\begin{enumerate}[label=($\spadesuit$\arabic*)]
\item $\mathcal{M}^{alg}_{l}$ is an $(l+1)$-category
\item the functor $\pi_{0}: \mathcal{M}^{alg}_{0} \rightarrow \ComodE$ is an equivalence
\item if $M \in \mathcal{M}^{alg}_{l-1}$, then there exists an obstruction in $\Ext_{E_{*}E}^{l+2, l}(u_{0} M, u_{0} M)$ which vanishes if and only if $M$ can be lifted to an algebraic potential $l$-stage, that is, there exists $\widetilde{M} \in \mathcal{M}^{alg}_{l}$ such that $u_{l-1} \widetilde{M} \simeq M$
\item for any $M, N \in \mathcal{M}^{alg}_{l}$ with $l \geq 1$ there exists a fibre sequence 

\begin{center}
$\map_{\mathcal{M}^{alg}_{l}}(M, N) \rightarrow \map_{\mathcal{M}_{l-1}^{alg}}(u_{l-1} M, u_{l-1}N) \rightarrow \map_{\dcat(\ComodE)}(u_{0} M, \Sigma^{l+1} u_{0} N [-l])$,
\end{center}

\item the functors $u_{l}: \mathcal{M}^{alg}_{\infty} \rightarrow \mathcal{M}^{alg}_{l}$ induce an equivalence $\mathcal{M}^{alg}_{\infty} \simeq \varprojlim \mathcal{M}^{alg}_{l}$

\item $\dcat(E_{*}E)\rightarrow \mathcal{M}^{alg}_{\infty}$ is an equivalence along which the functor $u_{0}: \mathcal{M}_{\infty} \rightarrow \mathcal{M}_{0}$ can be identified with taking homology $\dcat(E_{*}E) \rightarrow \ComodE$
\end{enumerate}
\end{prop}

\begin{proof}
This is the same as the topological case, which was \cite{moduli_problems_for_structured_ring_spectra}[3.3.2, 3.3.4, 3.3.5], but easier, since one doesn't have to construct $\synspectra$. The property $(\spadesuit 5)$, which requires Postnikov convergence, follows from \textbf{Remark \ref{rem:dcomode_is_left_complete}}. Again, a complete account will follow in \cite{abstract_gh_theory}.
\end{proof}
Note that the properties $(\spadesuit 1) - (\spadesuit 5)$ of \textbf{Proposition \ref{prop:properties_of_the_algebraic_gh_tower}} satisfied by the algebraic Goerss-Hopkins tower are exactly the same as those satisfied by the topological one, the only difference is $(\spadesuit 6)$, since $\mathcal{M}^{top}_{\infty} \simeq \spectra_{E}$, but in the algebraic case $\mathcal{M}^{alg}_{\infty} \simeq \dcat(E_{*}E)$. 

\begin{rem}
Taken together, \textbf{Proposition \ref{prop:properties_of_the_topological_gh_tower}} and \textbf{Proposition \ref{prop:properties_of_the_algebraic_gh_tower}} give an intuitive explanation as to why the $\infty$-categories $\spectra_{E}$ and $\dcat(E_{*}E)$ as formally analogous - they both arise as limits of towers satisfying a string of identical properties. 

Note that the existence of such towers also implies that mapping spaces in either of these can be computed by an Adams spectral sequence whose second page is given by $\Ext$-groups in $E_{*}E$-comodules. This is the original insight of Franke which made him focus on $\dcat(E_{*}E)$ \cite{franke1996uniqueness}.
\end{rem}
The following theorem, which is the main result of this section, says that in the case of finite homological dimension, the layers of the Goerss-Hopkins towers stabilize in a strong sense. 

\begin{thm}
\label{thm:both_elocal_cat_and_periodic_derived_cat_have_equivalent_homotopy_categories_to_moduli_of_high_potential_stages}
Let $p > n+1$. Then, $\mathcal{M}_{\infty}^{top} \rightarrow \mathcal{M}^{top}_{l}$ and $\mathcal{M}_{\infty}^{alg} \rightarrow \mathcal{M}_{l}^{alg}$ induce equivalences between homotopy $k$-categories, where $k= (l +1 - n^{2} - n)$. In particular, $h_{k} \spectra_{E} \simeq h_{k} \mathcal{M}^{top}_{l}$ and $h_{k} \dcat(E_{*}E) \simeq h_{k}\mathcal{M}^{alg}_{l}$. 
\end{thm}

\begin{proof}

Before proceeding with the proof, recall that by \textbf{Theorem \ref{thm:injective_dimension_of_ee_comodules}}, if $p > n+1$, then the abelian category $\ComodE$ is of finite homological dimension equal to $n^{2}+n$; in other words, all $\Ext$-groups between $E_{*}E$-comodules vanish above that line. 

The proof will only use the properties $(\spadesuit 3) - (\spadesuit 5)$ of \textbf{Proposition \ref{prop:properties_of_the_topological_gh_tower}} and \textbf{Proposition \ref{prop:properties_of_the_algebraic_gh_tower}}, with the second part following immediately from $(\spadesuit 6)$. Thus, let us write $\mathcal{M}_{l}$ for either the topological or algebraic $\infty$-category of potential $l$-stages.

Observe that by property $(\spadesuit 5)$, it is enough to show that $\mathcal{M}_{l^\prime+1} \rightarrow \mathcal{M}_{l^\prime}$ is an equivalence on homotopy $k$-categories for all $l^\prime \geq l$. This only makes sense when $k \geq 1$, so that we can assume that $l \geq n^{2}+n$. 

By property $(\spadesuit 3)$, we have an obstruction to lifting an object $X \in \mathcal{M}_{l^\prime}$ to $\mathcal{M}_{l^\prime+1}$ lying in the group $\Ext^{l^\prime+3, l^\prime+1}_{E_{*}E}(u_{0}X, u_{0}X)$. Since $l^\prime+3 \geq l+3 > n^{2}+n$, the obstruction necessarily vanishes and we deduce that $\mathcal{M}_{l^\prime+1} \rightarrow \mathcal{M}_{l^\prime}$ is essentially surjective. 

It is now enough to show that $\map_{\mathcal{M}_{l^\prime+1}}(X, Y) \rightarrow \map_{\mathcal{M}_{l^\prime}}(u_{l^\prime}X, u_{l^\prime}Y)$ is a $k$-connected map of spaces for any $X, Y \in \mathcal{M}_{l^\prime+1}$, since then it is an equivalence after $(k-1)$-truncation. By property $(\spadesuit 4)$, we have a fibre sequence 

\begin{center}
$\map_{\mathcal{M}_{l^\prime+1}}(X, Y) \rightarrow \map_{\mathcal{M}_{l^\prime}}(u_{l^\prime}X, u_{l^\prime}Y) \rightarrow  \map_{\dcat(\ComodE)}(u_{0} X, \Sigma^{l^\prime+2} u_{0} Y [-l^\prime-1])$.
\end{center}
Since $\pi_{s} \map_{\dcat(\ComodE)}(u_{0} X, \Sigma^{l^\prime+2} u_{0} Y [-l^\prime-1]) \simeq \Ext_{E_{*}E}^{l^\prime+2-s, -l^\prime-1}(u_{0} X, u_{0}Y)$, the homological dimension of $\ComodE$ implies that the base of the above fibre sequence is $(l^\prime+1-n^{2}-n)$-connected. However, $l^\prime+1- n^{2}-n \geq l+1 -n^{2}-n = k$, which gives the needed result. 
\end{proof}

\section{The Bousfield adjunction}
\label{section:the_bousfield_adjunction}

In this section we show that the Bousfield splitting induces adjunctions between the $\infty$-categories $\dcat(\ComodE)$ and $\Mod_{\monunitt{l}}(\synspectra)$ for small values of $l$. We then prove that these adjunctions are monadic, a first step in establishing an algebraic description of $\Mod_{\monunitt{l}}(\synspectra)$. 

The idea is as follows. Assume that $2p - 2 > n^2+n$, and let $l \leq 2p-3-n^{2}-n$. Under these assumptions, we have constructed the Bousfield splitting functor $\beta: \ComodE \rightarrow h_{l+1}\spectra_{E}$ which gives a partial inverse to the homology functor, see \textbf{Definition \ref{defin:bousfield_splitting}}. 

Consider the composite $(\monunitt{l} \otimes -) \circ \nu: \spectra_{E} \rightarrow \Mod_{\monunitt{l}}(\synspectra)$ of the synthetic analogue construction and the free module functor. The essential image of this composite is contained in the subcategory $\mathcal{M}_{l}$ of potential $l$-stages, which is an $(l+1)$-category. We deduce that there's an induced functor from $h_{l+1}(\spectra_{E})$, which by abuse of notation we will also denote $(\monunitt{l} \otimes -) \circ \nu$. This induced functor can then be composed with $\beta$; we will use this composite to construct the needed adjunction. 

We would like our adjunction to be symmetric monoidal, so we have to be a little bit more careful. Recall that we have shown that when restricted to the subcategory of comodules which are projective over $E_{*}$, the Bousfield splitting is symmetric monoidal. Because of this, it will be convenient for us to restrict to comodules which are projective.

\begin{notation}
We denote the category of $E_{*}E$-comodules which are finitely generated and projective over $E_{*}$ by $\ComodE^{fp}$. 
\end{notation}
The category $\ComodE^{fp}$ should be regarded as a convenient choice of generators for the connective derived $\infty$-category of comodules. More precisely, the derived $\infty$-category is freely generated under colimits, as a symmetric monoidal $\infty$-category, by the image of the inclusion $\ComodE^{fp} \hookrightarrow \dcat(\ComodE)_{\geq 0}$, subject to the relations that the latter preserves direct sums and takes hypercovers to colimit diagrams, see \cite{pstrkagowski2018synthetic}[2.62, A.23]. 

\begin{lemma}
The functor $(\monunitt{0} \otimes -) \circ \nu \circ \beta: \ComodE^{fp} \rightarrow \Mod_{\monunitt{0}}(\synspectra)$ is symmetric monoidal. 
\end{lemma}

\begin{proof}
The restriction of $\beta$ to projective comodules is symmetric monoidal by \textbf{Theorem \ref{thm:bousfield_splitting_is_symmetric_monoidal}}, and $\nu$ is symmetric monoidal when restricted to spectra with projective homology by \cite{pstrkagowski2018synthetic}[4.24]. We deduce that $(\monunitt{0} \otimes -) \circ \nu \circ \beta$ is a composite of symmetric monoidal functors, which ends the argument. 
\end{proof}

\begin{lemma}
\label{lemma:bousfield_functor_extended_to_zero_modules_equivalent_to_the_inclusion}
The functor $(\monunitt{0} \otimes -) \circ \nu \circ \beta: \ComodE^{fp} \rightarrow \Mod_{\monunitt{0}}(\synspectra)$ uniquely extends to a symmetric monoidal equivalence $\dcat(\ComodE)_{\geq 0} \simeq \Mod_{\monunitt{0}}(\synspectra)$. 
\end{lemma}

\begin{proof}
Note that by \cite{pstrkagowski2018synthetic}[4.54], the $\infty$-category $\Mod_{\monunitt{0}}(\synspectra)$ is equivalent to the connective derived $\infty$-category, we only have to verify that the given functor uniquely extends to one. 

By \cite{pstrkagowski2018synthetic}[4.22], we have $\monunit_{0} \otimes \nu X \simeq E_{*}X$ when considered as an object of $\synspectra^{\heartsuit} \simeq \ComodE$. Since $E_{*} (\beta M) \simeq M$ for any comodule $M$ by the construction of the Bousfield splitting, we deduce that the composite $(\monunitt{0} \otimes -) \circ \nu \circ \beta$ can be identified with the inclusion of $\ComodE^{fp}$ into the heart. The unique extension to an equivalence then exists by \cite{pstrkagowski2018synthetic}[2.62, A.23],  
\end{proof}

\begin{prop}
\label{prop:composite_of_bousfield_splitting_and_free_module_extends_to_a_cocontinuous_functor}
The functor $(\monunitt{l} \otimes -) \circ \nu \circ \beta: \ComodE^{fp} \rightarrow \Mod_{\monunitt{l}}(\synspectra)$ uniquely extends to a symmetric monoidal, cocontinuous functor $\dcat(\ComodE)_{\geq 0} \rightarrow \Mod_{\monunitt{l}}(\synspectra)$. 
\end{prop}
\begin{proof}
By \cite{pstrkagowski2018synthetic}[2.62, A.23], it is enough to check that the given functor preserves direct sums and takes hypercovers to colimit diagrams. Since extension of scalars along $\monunitt{l} \rightarrow \monunitt{0}$ is cocontinuous and conservative, we deduce that it is enough to check that 

\begin{center}
$\monunitt{0} \otimes_{\monunitt{l}} ((\monunitt{l} \otimes -) \circ \nu \circ \beta) \simeq (\monunitt{0} \otimes -) \circ \nu \circ \beta$
\end{center}
has these two properties, which is immediate from \textbf{Lemma \ref{lemma:bousfield_functor_extended_to_zero_modules_equivalent_to_the_inclusion}}.
\end{proof}

Notice that since the extension constructed in \textbf{Proposition \ref{prop:composite_of_bousfield_splitting_and_free_module_extends_to_a_cocontinuous_functor}} is a cocontinuous functor between presentable $\infty$-categories, it necessarily has a right adjoint. 

\begin{defin}
\label{defin:bousfield_adjunction}
We denote the unique cocontinuous extension of the composite $(\monunitt{l} \otimes -) \circ \nu \circ \beta$ by $\beta^{*}: D(\ComodE) \rightarrow \Mod_{\monunitt{l}}(\synspectra)$ and its right adjoint by $\beta_{*}: \Mod_{\monunitt{l}}(\synspectra) \rightarrow \dcat(\ComodE)$.  We call the adjunction $\beta^{*} \dashv \beta_{*}$ the \emph{Bousfield adjunction}. 
\end{defin}
The rest of the section will be devoted to the proof that the Bousfield adjunction is monadic. We will later see that this implies that $\Mod_{\monunitt{l}}(\synspectra)$ can be described as an $\infty$-category of modules in $\dcat(\ComodE)_{\geq 0}$, and that the Bousfield adjunction can then be identified with extension and restriction of scalars. This motivates our choice of notation. 

By Barr-Beck-Lurie, to verify that $\beta^{*} \dashv \beta_{*}$ is monadic, we have to check that $\beta_{*}$ preserves certain kinds of geometric realizations, and that it is conservative. We begin with the former, in fact, we will show that $\beta_{*}$ is cocontinuous. 

\begin{lemma}
\label{lemma:bousfield_adjunction_induces_an_equivalence_on_hearts}
The induced adjunction $(-)_{\leq 0} \circ \beta^{*} \dashv \beta_{*}: \dcat(\ComodE)^{\heartsuit}_{\geq 0} \rightleftarrows \Mod_{\monunitt{l}}(\synspectra)^{\heartsuit}$ between the hearts is an adjoint equivalence. 
\end{lemma}

\begin{proof}
Since $\monunitt{0} \otimes_{\monunitt{l}} ((\monunitt{l} \otimes -) \circ \nu \circ \beta \simeq (\monunitt{0} \otimes -) \circ \nu \circ \beta$, we see that \textbf{Lemma \ref{lemma:bousfield_functor_extended_to_zero_modules_equivalent_to_the_inclusion}} implies that the composite adjunction 

\begin{center}
$D(\ComodE)_{\geq 0} \rightleftarrows \Mod_{\monunitt{l}}(\synspectra) \rightleftarrows \Mod_{\monunitt{0}}(\synspectra)$
\end{center}
is an adjoint equivalence, in particular an adjoint equivalence between the hearts. Since the latter is also true for the right adjunction, we deduce that it also holds for the left one, which is what we wanted to show. 
\end{proof}

\begin{lemma}
\label{lemma:right_adjoint_in_bousfield_adjunction_preserves_connected_objects}
If $X \in \Mod_{\monunitt{l}}(\synspectra)$ is $k$-connected, then so is $\beta_{*}X \in \dcat(\ComodE)_{\geq 0}$.
\end{lemma}

\begin{proof}
Since the composite $(\monunitt{0} \otimes_{\monunitt{l}} -) \circ \beta^{*}$ is an equivalence by \textbf{Lemma \ref{lemma:bousfield_functor_extended_to_zero_modules_equivalent_to_the_inclusion}}, so must be its right adjoint, which is given by the restriction of scalars along $\monunitt{l} \rightarrow \monunitt{0}$ followed by $\beta_{*}$. Since restriction of scalars preserves and reflects connectivity, it follows that the result is true if $X$ admits a structure of a $\monunitt{0}$-module.

Let us prove by induction that $\beta_{*} (\monunitt{i} \otimes_{\monunitt{l}} X)$ is $k$-connected for $i$ in the range $l \geq i \geq 0$. Since the base case is covered above, we assume that $i > 0$. We have $\pi_{*} \monunitt{l} \simeq E_{*}[\tau] / (\tau^{l+1})$, and so there exists a cofibre sequence $\monunitt{i} \rightarrow \monunitt{i-1} \rightarrow \Sigma^{i+1} \monunitt{0}[-i] $, which after tensoring with $X$ and applying $\beta_{*}$ yields a fibre sequence

\begin{center}
$\beta_{*} (\monunitt{i} \otimes_{\monunitt{l}} X) \rightarrow \beta_{*}(\monunitt{i-1} \otimes_{\monunitt{l}} X) \rightarrow \beta_{*} (\Sigma^{i+1} \monunitt{0}[-i] \otimes_{\monunitt{l}} X)$.
\end{center}
Then, by the inductive assumption both the right and middle objects are $k$-connected, and from the long exact sequence of homotopy we deduce that the same must be true for the fibre. This ends the argument, since for $i = l$ we have $\monunitt{i} \otimes _{\monunitt{l}} X \simeq X$. 
\end{proof}

\begin{lemma}
\label{lemma:right_adjoint_in_bousfield_adjunction_preserves_homotopy_groups}
The right adjoint $\beta_{*}$ commutes with $t$-structure homotopy groups in the sense that the natural map $\pi_{k} \beta_{*} X \rightarrow \beta_{*} \pi_{k} X$ is an isomorphism for each $k \geq 0$ and $X \in \Mod_{\monunitt{l}}(\synspectra)$.
\end{lemma}

\begin{proof}
Let us first describe the natural map, where we first recall that if $X$ is an object of a Grothendieck prestable $\infty$-category, then $\pi_{k} X = (\Omega^{k} X)_{\leq 0}$, considered as an object of the heart. Since $\beta_{*}$ is continuous, it commutes with loops and so it is enough to define the natural map for $k = 0$. In this case, by the natural map we mean the one induced by $\beta_{*} X \rightarrow \beta_{*} X_{\leq 0}$. 

Consider the cofibre sequence $X \rightarrow X_{\leq 0} \rightarrow C$ and observe that the cofibre $C$ is $1$-connected. Applying $\beta_{*}$ we obtain a fibre sequence

\begin{center}
$\beta_{*}X \rightarrow \beta_{*}X_{\leq 0} \rightarrow \beta_{*}C$,
\end{center}
where $\beta_{*}C$ is also $1$-connected by \textbf{Lemma \ref{lemma:right_adjoint_in_bousfield_adjunction_preserves_connected_objects}}. It follows from the long exact sequence of homotopy that $(\beta_{*} X)_{\leq 0} \simeq \beta_{*} X_{\leq 0}$, which is what we wanted to show. 
\end{proof} 

\begin{lemma}
\label{lemma:bousfield_right_adjoint_is_cocontinuous}
The right adjoint $\beta_{*}: \Mod_{\monunitt{l}}(\synspectra) \rightarrow \dcat(\ComodE)_{\geq 0}$ is cocontinuous. 
\end{lemma}

\begin{proof}
It's enough to show that $\beta_{*}$ commutes with cofibre sequences and direct sums. Notice that both $\infty$-categories are separated Grothendieck prestable, so that $\pi_{k}(\bigoplus X_{\alpha}) \simeq \bigoplus \pi_{k} X_{\alpha}$ and direct sums are detected by their homotopy groups. Then, the statement for direct sums follows from the fact that $\beta_{*}: \Mod_{\monunitt{l}}(\synspectra)^{\heartsuit} \rightarrow \dcat(\ComodE)^{\heartsuit}_{\geq 0}$ is an equivalence by \textbf{Lemma \ref{lemma:bousfield_adjunction_induces_an_equivalence_on_hearts}}, in particular commutes with direct sums, and \textbf{Lemma \ref{lemma:right_adjoint_in_bousfield_adjunction_preserves_homotopy_groups}}. 

To see that $\beta_{*}$ preserves cofibres, notice that in a Grothendieck prestable $\infty$-category, a fibre sequence $X \rightarrow Y \rightarrow Z$ is cofibre if and only if $\pi_{0} Y \rightarrow \pi_{0} Z$ is an epimorphism. Given such a cofibre sequence in $\Mod_{\monunitt{l}}(\synspectra)$, it is clear that $\beta_{*} X \rightarrow \beta_{*} Y \rightarrow \beta_{*}Z$ is fibre, since $\beta_{*}$ is continuous, and $\pi_{0} \beta_{*} Y \rightarrow \pi_{0} \beta_{*} Z$ is an epimorphism by \textbf{Lemma \ref{lemma:bousfield_adjunction_induces_an_equivalence_on_hearts}} and \textbf{Lemma \ref{lemma:right_adjoint_in_bousfield_adjunction_preserves_homotopy_groups}}. This ends the argument. 
\end{proof}

\begin{thm}
\label{thm:bousfield_adjunction_is_monadic}
The adjunction $\beta^{*} \dashv \beta_{*}: \dcat(\ComodE)_{\geq 0} \rightleftarrows Mod_{\monunitt{l}}(\synspectra)$ is monadic. 
\end{thm}

\begin{proof}
By Barr-Beck-Lurie \cite{higher_algebra}[4.7.4.5], it is enough to verify that the right adjoint $\beta_{*}$ is conservative and preserves certain geometric realizations. The latter is clear, since it is even cocontinuous by \textbf{Lemma \ref{lemma:bousfield_right_adjoint_is_cocontinuous}}. 

To see that $\beta_{*}$ is conservative, notice that it commutes with taking homotopy groups by \textbf{Lemma \ref{lemma:right_adjoint_in_bousfield_adjunction_preserves_homotopy_groups}} and \textbf{Lemma \ref{lemma:bousfield_adjunction_induces_an_equivalence_on_hearts}}, and that homotopy groups detect equivalences in both $\infty$-categories. 
\end{proof}

\section{An algebraic model for the $E$-local homotopy category}

In this section we finish the proof of the main result of this note, showing that at large primes we have an equivalence of homotopy categories $h \spectra_{E} \simeq h\dcat(E_{*}E)$ between the $\infty$-category of $E$-local spectra and the derived $\infty$-category of differential $E_{*}E$-comodules.

Recall that in $\S\ref{section:goerss_hopkins_theory}$ we have used Goerss-Hopkins theory to associate to each of $\spectra_{E}$ and $\dcat(E_{*}E)$ a tower of $\infty$-categories of what we called topological, respectively algebraic, potential $l$-stages. Then, we've shown that if $p > n+1$, so that the abelian category of $E_{*}E$-comodules has finite homological dimension, the functors $\spectra_{E} \rightarrow \mathcal{M}^{top}_{l}$ and $\dcat(E_{*}E) \rightarrow \mathcal{M}^{alg}_{l}$ induce equivalences on homotopy $k$-categories, where $k = (l +1 - n^{2} - n)$. 

Thus, we see that the proof of the main result can be reduced to the task of comparing the homotopy categories of the $\infty$-categories $\mathcal{M}_{l}^{top}$ and $\mathcal{M}_{l}^{alg}$ of potential $l$-stages. This will be done in this section, in fact, we will show something stronger, namely that the two are, for $l$ in a certain explicit range, equivalent as $\infty$-categories. To do so, we will employ the Bousfield adjunction studied in \S\ref{section:the_bousfield_adjunction}.

\begin{assumption}
\label{assumption:l_low_enough_so_that_bousfield_adjunction_exists}
In this section, we assume that $2p - 2 > n^2+n$ and $l \leq  2p-2-n^{2}-n-1$.
\end{assumption}
Under the above assumption we have constructed the Bousfield adjunction 

\begin{center}
$\beta^{*} \dashv \beta_{*}: \dcat(\ComodE)_{\geq 0} \rightleftarrows Mod_{\monunitt{l}}(\synspectra)$, 
\end{center}
see \textbf{Definition \ref{defin:bousfield_adjunction}}. By \textbf{Proposition \ref{prop:composite_of_bousfield_splitting_and_free_module_extends_to_a_cocontinuous_functor}}, $\beta^{*}$ is symmetric monoidal, and it follows formally that the right adjoint $\beta_{*}$ is lax symmetric monoidal. We deduce that $\beta_{*} \monunitt{l}$ is canonically a commutative algebra in the derived $\infty$-category. 

\begin{lemma}
\label{lemma:homotopy_of_beta_monunittl}
Let $\tau: \monunitt{l}[-1] \rightarrow \monunitt{l}$ be induced from the shift map of the monoidal unit of synthetic spectra, so that $\pi_{*} \monunitt{l} \simeq (\pi_{0} \monunit)[\tau] / (\tau)^{l+1}$ in the sense that $\pi_{k} \monunitt{l} \simeq (\pi_{0} \monunit)[-k]$ for $k \leq l$, zero otherwise, and $\tau$ acts isomorphically on $\pi_{*} \monunitt{l}$ through degrees up to $l$. Then, $\pi_{*} \beta_{*} \monunitt{l} \simeq E_{*}[\beta_{*} \tau] / (\beta_{*} \tau)^{l+1}$ in the same sense.
\end{lemma}

\begin{proof}
Let us recall that the shift map $\tau$ in synthetic spectra is defined in \cite{pstrkagowski2018synthetic}[4.27] and has the above property by \cite{pstrkagowski2018synthetic}[4.61]. 
Since $\beta_{*}$ commutes with taking homotopy groups by \textbf{Lemma \ref{lemma:right_adjoint_in_bousfield_adjunction_preserves_homotopy_groups}}, we deduce that $\pi_{*} \beta_{*} \monunitt{l} \simeq (\beta_{*} \pi_{0} \monunitt{l})[\beta_{*} \tau](\beta_{*} \tau)^{l+1}$. Then, \textbf{Lemma \ref{lemma:bousfield_adjunction_induces_an_equivalence_on_hearts}} implies that $\beta_{*} \pi_{0} \monunitt{l} \simeq E_{*}$, in fact we have a canonical such isomorphism induced by the unit of the algebra $\beta_{*} \monunitt{l}$. 
\end{proof}

\begin{cor}
\label{cor:beta_beta_estar_isomorphic_to_truncation_of_free_associative_algebra}
Consider the composite $\Sigma E_{*} [-1] \rightarrow \Sigma \beta_{*} \monunitt{l} [-1] \rightarrow \beta_{*} \monunitt{l}$, where the first map is the unit and the second is $\beta^{*} \tau$. Then, the induced map $F(\Sigma E_{*}[-1])\rightarrow \beta_{*} \monunitt{l}$ from the free associative algebra on $\Sigma E_{*}[-1]$ induces an equivalence $F(\Sigma E_{*}[-1])_{\leq l} \simeq \beta_{*} \monunitt{l}$.
\end{cor}

\begin{proof}
It is immediate from \textbf{Lemma \ref{lemma:homotopy_of_beta_monunittl}} that the constructed map $F(\Sigma E_{*}[-1]) \rightarrow \beta_{*} \monunitt{l}$ is an isomorphism on homotopy in degrees up to $l$, proving the claim.
\end{proof}

Notie that even though \textbf{Lemma \ref{lemma:homotopy_of_beta_monunittl}} roughly says that the homotopy of $\monunitt{l}$ and $\beta_{*} \monunitt{l}$ is the same, \textbf{Corollary \ref{cor:beta_beta_estar_isomorphic_to_truncation_of_free_associative_algebra}} shows that this has stronger consequence for the latter than for the former, the difference being that the unit $E_{*}$ of $\dcat(\ComodE)_{\geq 0}$ is discrete.

\begin{thm}
\label{thm:right_adjoint_to_betastar_lifts_to_a_symmon_equiv}
The functor $\beta_{*}: \Mod_{\monunitt{l}}(\synspectra) \rightarrow \dcat(\ComodE)_{\geq 0}$ lifts to a symmetric monoidal adjoint equivalence $\gamma^{*} \dashv \gamma_{*}: \Mod_{\beta^{*} \monunitt{l}}(\dcat(\ComodE)_{\geq 0}) \rightleftarrows  \Mod_{\monunitt{l}}(\synspectra)$.
\end{thm}
 
\begin{proof}
We have a span of symmetric monoidal $\infty$-categories and lax symmetric monoidal functors 

\begin{center}
$\Mod_{\monunitt{l}}(\synspectra) \leftarrow \Mod_{\monunitt{l}}(\Mod_{\monunitt{l}}(\synspectra)) \rightarrow \Mod_{\beta_{*} \monunitt{l}}(\dcat(\ComodE)_{\geq 0})$, 
\end{center}
where the right one is induced by $\beta_{*}$ and the left one is the forgetful functor. Since $\monunitt{l}$ is the monoidal unit of $\Mod_{\monunitt{l}}(\synspectra)$, the left one is an equivalence. Choosing an inverse and composing we obtain a lax symmetric monoidal functor $\gamma_{*}: \Mod_{\monunitt{l}}(\synspectra) \rightarrow \Mod_{\beta_{*} \monunitt{l}}(\dcat(\ComodE)_{\geq 0})$. 

We claim that $\gamma_{*}$ is an equivalence. By construction, it fits into a commutative diagram

\begin{center}
	\begin{tikzpicture}
		\node (TL) at (0, 1) {$ \Mod_{\monunitt{l}}(\synspectra) $};
		\node (TR) at (4, 1) {$ \Mod_{\beta_{*} \monunitt{l}}(\dcat(\ComodE)_{\geq 0}) $};
		\node (BM) at (2, 0) {$ \dcat(\ComodE)_{\geq 0}$};
		
		\draw [->] (TL) to (TR);
		\draw [->] (TL) to (BM);
		\draw [->] (TR) to (BM);
	\end{tikzpicture},
\end{center}
where the right arrow is the forgetful functor. Both vertical arrows are cocontinuous and present their source as monadic over $\dcat(\ComodE)_{\geq 0}$, the left one by \textbf{Lemma \ref{lemma:bousfield_right_adjoint_is_cocontinuous}} and \textbf{Theorem \ref{thm:bousfield_adjunction_is_monadic}}. 

Thus, by \cite[4.7.4.16]{higher_algebra} it is enough to verify that that for any $M \in \dcat(\ComodE)_{\geq 0}$, the map $\beta_{*} \monunitt{l} \otimes_{E_{*}} M \rightarrow \beta_{*} \beta^{*} M$ induced by the unit is an equivalence. Since both the source and the target of this map are clearly cocontinuous in $M$, we can assume that $M$ is a comodule which is finitely generated and projective over $E_{*}$. 

We want to show that the map $\pi_{*} (\beta_{*} \monunitt{l} \otimes_{E_{*}} M) \rightarrow \pi_{*} \beta_{*} \beta^{*} M$ is an isomorphism. Since $M$ is finitely generated projective, by construction we have $\beta^{*}(M) = \monunitt{l} \otimes \nu \beta (M)$ and we deduce that $\pi_{*} \beta^{*} M \simeq \pi_{*} \monunitt{l} \otimes _{\pi_{0} \monunitt{l}} \pi_{0} \beta^{*}M$ by the calculation of the homotopy of synthetic analogues \cite{pstrkagowski2018synthetic}[4.61]. Then, \textbf{Lemma \ref{lemma:right_adjoint_in_bousfield_adjunction_preserves_homotopy_groups}} implies that 

\begin{center}
$\pi_{*} \beta_{*} \beta^{*} M \simeq \beta_{*} \pi_{*} \beta^{*} M \simeq \beta_{*} \pi_{*} \monunitt{l} \otimes _{\beta_{*} \pi_{0} \monunitt{l}} \beta_{*} \pi_{0} \beta^{*} M \simeq \pi_{*} \beta_{*} \monunitt{l} \otimes _{E_{*}} M$,
\end{center}
where the isomorphisms $\beta_{*} \pi_{0} \monunitt{l} \simeq E_{*}$ and $\beta_{*} \pi_{0} \beta^{*} M \simeq M$ follow from \textbf{Lemma \ref{lemma:bousfield_adjunction_induces_an_equivalence_on_hearts}}, which says that $\beta^{*} \dashv \beta_{*}$ induces an adjoint equivalence on the hearts. 

We deduce that the homotopy groups of $\beta_{*} \monunitt{l} \otimes _{E_{*}} M$ and $\beta_{*} \beta^{*} M$ are abstractly isomorphic, in fact coincide with a free $\pi_{*} \beta_{*} \monunitt{l}$-module. Thus, to prove that $\pi_{*} (\beta_{*} \monunitt{l} \otimes_{E_{*}} M) \rightarrow \pi_{*} \beta_{*} \beta^{*} M$ is an isomorphism it is enough to show that it is an isomorphism in degree zero, which is again a consequence of \textbf{Lemma \ref{lemma:bousfield_adjunction_induces_an_equivalence_on_hearts}}. This ends the proof that $\gamma_{*}$ is an equivalence.

It follows that we can choose a left adjoint $\gamma^{*}: \Mod_{\beta_{*} \monunitt{l}} (\dcat(\ComodE)_{\geq 0}) \rightarrow \Mod_{\monunitt{l}}(\synspectra) $, and since $\gamma_{*}$ is lax symmetric monoidal we deduce that $\gamma^{*}$ has a canonical structure of an oplax symmetric monoidal functor. We claim that it is in fact symmetric monoidal, which shows that the same is true for its right adjoint, ending the proof of the theorem.

We have to verify that for any $M, N \in \Mod_{\beta_{*} \monunitt{l}}(\dcat(\ComodE)_{\geq 0})$, the structure map 
$\gamma^{*} (M \otimes_{\beta_{*} \monunitt{l}} N) \rightarrow \gamma^{*} M \otimes _{\monunitt{l}} \gamma^{*} N$ is an equivalence. Since by construction $\gamma_{*}$ lifts $\beta_{*}$, by passing to left adjoints we see that $\gamma^{*} (\beta_{*} \monunitt{l} \otimes _{E_{*}} M^{\prime}) \simeq \beta^{*} M^{\prime}$ for any $M^\prime \in \dcat(\ComodE)_{\geq 0}$. Then, because $\beta^{*}$ is symmetric monoidal, we deduce that the structure map of $\gamma^{*}$ is an equivalence when $M \simeq \beta_{*} \monunitt{l} \otimes _{E_{*}} M^\prime$ and $N \simeq \beta_{*} \monunitt{l} \otimes _{E_{*}} N^\prime$ for $M^\prime, N^\prime \in \dcat(\ComodE)_{\geq 0}$.

We now consider the collection of those $M \in \Mod_{\beta_{*} \monunitt{l}}(\dcat(\ComodE)_{\geq 0})$ such that 

\begin{center}
$\gamma^{*} (M \otimes _{\beta^{*} \monunitt{l}} \beta_{*} \monunitt{l} \otimes _{E_{*}} N^\prime) \rightarrow \gamma^{*} M \otimes_{\monunitt{l}} \gamma^{*} (\beta_{*} \monunitt{l} \otimes _{E_{*}} N^\prime)$
\end{center}
is an equivalence for all $N^\prime \in \dcat(\ComodE)_{\geq 0}$. By the previous observation, this holds when $M$ itself is of the form $\beta_{*} \monunitt{l} \otimes _{E_{*}} M^\prime$, and since the collection of $M$ satisfying the above property is clearly closed under colimits, we deduce that all $M \in \Mod_{\beta_{*} \monunitt{l}}(\dcat(\ComodE)_{\geq 0})$ belong to it. 

Similarly, by considering colimits in the other variable we see that if $M$ has the above property, then $\gamma^{*} (M \otimes _{\beta_{*} \monunitt{l}} N) \rightarrow \gamma^{*} M \otimes_{\monunitt{l}} \gamma^{*} N$ is an equivalence for all $N \in \Mod_{\beta_{*} \monunitt{l}}(\dcat(\ComodE)_{\geq 0})$, which is what we wanted to show.
 \end{proof}
 
Recall that in \textbf{Definition \ref{defin:periodicity_algebra}} we have introduced an algebra $P$ in $\dcat(\ComodE)_{\geq 0}$ with homotopy groups $\pi_{*} P \simeq E_{*}[\tau]$, which has the property that the derived $\infty$-category $\dcat(E_{*}E)$ of differential comodules can be identified with a certain subcategory of connective $P$-modules. This observation was then used to construct the Goerss-Hopkins tower of $\dcat(E_{*}E)$, which was informally induced from the Postnikov tower of $P$. 

\begin{cor}
\label{cor:modules_over_monunitt_and_truncation_of_connective_periodic_unit_equivalent_as_infty_cats}
There exists an equivalence $\delta^{*}: \Mod_{\monunitt{l}}(\synspectra) \rightarrow \Mod_{P_{\leq l}}(\dcat(\ComodE)_{\geq 0})$ which fits into a commutative diagram

\begin{center}
	\begin{tikzpicture}
		\node (TL) at (0, 1.4) {$ \Mod_{\monunitt{l}}(\synspectra) $};
		\node (TR) at (4, 1.4) {$ \Mod_{P_{\leq l}}(\dcat(\ComodE)_{\geq 0}) $};
		\node (BL) at (0, 0) {$ \Mod_{\monunitt{0}}(\synspectra)  $};
		\node (BR) at (4, 0) {$ \Mod_{P_{\leq 0}}(\dcat(\ComodE)_{\geq 0}) $};
		
		\draw [->] (TL) to node[auto] {$ \delta^{*} $} (TR);
		\draw [->] (TL) to node[auto] {$ \monunitt{0} \otimes _{\monunitt{l}} - $} (BL);
		\draw [->] (TR) to node[right] {$ P_{\leq 0} \otimes _{P_{\leq l}} - $} (BR);
		\draw [->] (BL) to (BR);
	\end{tikzpicture}
\end{center}
in which both horizontal arrows are equivalences.
\end{cor}

\begin{proof}
In \textbf{Theorem \ref{thm:right_adjoint_to_betastar_lifts_to_a_symmon_equiv}} we have constructed $\gamma^{*}: \Mod_{\monunitt{l}}(\synspectra) \rightarrow Mod_{\beta_{*} \monunitt{l}}(\dcat(\ComodE)_{\geq 0})$, which is a symmetric monoidal equivalence of $\infty$-categories. This yields a commutative diagram 

\begin{center}
	\begin{tikzpicture}
		\node (TL) at (0, 1.4) {$ \Mod_{\monunitt{l}}(\synspectra) $};
		\node (TR) at (4.3, 1.4) {$ \Mod_{\beta_{*} \monunitt{l}}(\dcat(\ComodE)_{\geq 0}) $};
		\node (BL) at (0, 0) {$ \Mod_{\monunitt{0}}(\synspectra)  $};
		\node (BR) at (4.3, 0) {$ \Mod_{(\beta_{*} \monunitt{l}) _{\leq 0}}(\dcat(\ComodE)_{\geq 0}) $};
		
		\draw [->] (TL) to node[auto] {$ \gamma^{*} $} (TR);
		\draw [->] (TL) to node[auto] {$ \monunitt{0} \otimes _{\monunitt{l}} - $} (BL);
		\draw [->] (TR) to node[right] {$ (\beta_{*} \monunitt{l})_{\leq 0} \otimes _{\beta_{*} \monunitt{l}} - $} (BR);
		\draw [->] (BL) to (BR);
	\end{tikzpicture},
\end{center}
where the vertical arrows are obtained by tensoring with the discretization of the unit; clearly the bottom horizontal one is also an equivalence.

Then, as a combination of \textbf{Remark \ref{algebra_p_equivalent_to_the_free_algebra}} and \textbf{Corollary \ref{cor:beta_beta_estar_isomorphic_to_truncation_of_free_associative_algebra}}, we see that the two algebras $\beta_{*} \monunitt{l}$ and $P_{\leq l}$ are equivalent as associative algebras, both being equivalent to the truncation of the free associative algebra on $\Sigma E_{*}[-1]$. In fact, in both cases we have constructed preferred such equivalences, which in turn determines an equivalence $\beta^{*} \monunitt{l} \rightarrow P_{\leq l}$. We obtain a commutative diagram

\begin{center}
	\begin{tikzpicture}
		\node (TL) at (0, 1) {$ \beta_{*} \monunitt{l} $};
		\node (TR) at (2, 1) {$ P_{\leq l} $};
		\node (BL) at (0, 0) {$ (\beta_{*} \monunitt{l})_{\leq 0} $}; 
		\node (BR) at (2, 0) {$ P_{\leq 0} $};
		
		\draw [->] (TL) to (TR);
		\draw [->] (TL) to  (BL);
		\draw [->] (TR) to (BR);
		\draw [->] (BL) to (BR);
	\end{tikzpicture},
\end{center}
of associative algebras in $\dcat(\ComodE)_{\geq 0}$. Then, by taking the associated diagram of module $\infty$-categories and extension of scalars functors and pasting it with the one constructed above yields a square as in the statement of the theorem.
\end{proof}

\begin{rem}
The two $\infty$-categories appearing in the statement of \textbf{Corollary \ref{cor:modules_over_monunitt_and_truncation_of_connective_periodic_unit_equivalent_as_infty_cats}} both carry symmetric monoidal structures, but the equivalence $\delta^{*}$ we have constructed is not necessarily symmetric monoidal. 

By delving into the proof, one sees that the issue is that we know that $\beta_{*} \monunitt{l}$ and $P_{\leq l}$ are equivalent as associative algebras, which is enough to identify their $\infty$-categories of modules, but we don't know whether they're equivalent as commutative algebras.
\end{rem}

\begin{cor}
\label{cor:moduli_of_topological_and_algebraic_potential_l_stages_equivalent_for_low_l}
The equivalence $\delta^{*}: \Mod_{\monunitt{l}}(\synspectra) \rightarrow \Mod_{P_{\leq l}}(\dcat(\ComodE)_{\geq 0})$ restricts to an equivalence $\mathcal{M}_{l}^{top} \simeq \mathcal{M}_{l}^{alg}$ between the $\infty$-categories of topological and algebraic potential $l$-stages. 
\end{cor}

\begin{proof}
Recall that a topological potential $l$-stage is an $\monunitt{l}$-module $X$ in synthetic spectra such that $\monunitt{0} \otimes _{\monunitt{l}} X$ is discrete, and an algebraic potential $l$-stage $M$ is a $P_{\leq l}$-module in $\dcat(\ComodE)_{\geq 0}$ such that $P_{\leq 0} \otimes _{P_{\leq l}} M$ is discrete. Thus, the statement is immediate from \textbf{Corollary \ref{cor:modules_over_monunitt_and_truncation_of_connective_periodic_unit_equivalent_as_infty_cats}}.
\end{proof}
The following is the main result of this note. 

\begin{thm}
\label{thm:at_large_primes_homotopy_category_of_e_local_spectra_same_as_periodic_chain_complexes}
Let $E$ be a $p$-local Landweber exact homology theory of height $n$ and assume that $p > n^{2} + n + 1 + \frac{k}{2}$. Then, there exists an equivalence $h_{k} \spectra_{E} \simeq h_{k}\dcat(E_{*}E)$ between the homotopy $k$-categories of $\spectra_{E}$ and the derived $\infty$-category of $E_{*}E$.
\end{thm}

\begin{proof}
Let $l = n^{2}+n+k-1$, so that by \textbf{Theorem \ref{thm:both_elocal_cat_and_periodic_derived_cat_have_equivalent_homotopy_categories_to_moduli_of_high_potential_stages}}, the functors $\spectra_{E} \rightarrow \mathcal{M}^{top}_{l}$ and $\dcat(E_{*}E) \rightarrow \mathcal{M}^{alg}_{l}$ induce equivalences $h_{k}\spectra_{E} \simeq h_{k}\mathcal{M}_{l}^{top}$ and $h_{k}\dcat(E_{*}E) \simeq h_{k}\mathcal{M}_{l}^{alg}$ between homotopy $k$-categories.

Now suppose that $p > n^{2}+n+1+\frac{k}{2}$. This is the same as $2p \geq 2n^{2}+2n+2+k$, which implies $2p-2-n^{2}-n-1 \geq n^{2}+n+k-1$. Since $l = n^{2}+n+k-1$, we deduce that \textbf{Assumption \ref{assumption:l_low_enough_so_that_bousfield_adjunction_exists}} holds. Then, by \textbf{Corollary \ref{cor:moduli_of_topological_and_algebraic_potential_l_stages_equivalent_for_low_l}} there exists an equivalence $\mathcal{M}_{l}^{top} \simeq \mathcal{M}_{l}^{alg}$, which in particular implies $h_{k} \mathcal{M}_{l}^{top} \simeq h_{k} \mathcal{M}_{l}^{alg}$. This ends the proof.
\end{proof}

\begin{cor}
Let $p > n^{2}+n+2$. Then, there exists an equivalence $h \spectra_{E} \simeq h \dcat(E_{*}E)$ of triangulated categories. 
\end{cor}

\begin{proof}
Recall that $h \spectra_{E}$ and $\dcat(E_{*}E)$ are both canonically triangulated since they are homotopy categories of a stable $\infty$-category \cite[1.1.2.15]{higher_algebra}. By \textbf{Theorem \ref{thm:triangulated_structure_on_the_homotopy_cat_of_stable_cat_only_depends_on_homotopy_2_cat}}, which we prove in the appendix, this triangulated structure is in fact already determined by the homotopy $2$-category. Then, since $p > n^{2}+n+2$, \textbf{Theorem \ref{thm:at_large_primes_homotopy_category_of_e_local_spectra_same_as_periodic_chain_complexes}} implies that $h_{2} \spectra_{E} \simeq h_{2} \dcat(E_{*}E)$, ending the argument. 
\end{proof}

\FloatBarrier
\appendix

\section{Connectivity results for undercategories}

In this appendix we prove a connectivity result for mapping spaces in undercategories. This will be used in the next section to prove that the triangulated structure on the homotopy category of a stable $\infty$-category is already determined by the homotopy $2$-category. 

\begin{defin}
\label{defin:locally_n_connected_functor}
We say a functor $f: \ccat \rightarrow \dcat$ of $\infty$-categories is \emph{locally $n$-connected} if for any $c, c^\prime \in \ccat$, the induced morphism $\map_{\ccat}(c, c^\prime) \rightarrow \map_{\dcat}(f(c), f(c^\prime))$ is an $n$-connected map of spaces. 
\end{defin}
In more detail, we will prove that functors which are highly locally connected in the sense of \textbf{Definition \ref{defin:locally_n_connected_functor}} induce highly locally connected functors between undercategories, with effective bounds on the loss of connectivity, see \textbf{Proposition \ref{prop:connectivity_of_functors_between_undercategories}}.

\begin{lemma}
\label{lemma:connectivity_of_maps_between_pullbacks_of_spaces}
Suppose we have a natural transformation between spans of spaces 

\begin{center}
	\begin{tikzpicture}
		\node (TL) at (0, 1) {$ A $};
		\node (TM) at (1, 1) {$ C $};
		\node (TR) at (2, 1) {$ B $};
		\node (BL) at (0, 0) {$ A^\prime $};
		\node (BM) at (1, 0) {$ C^\prime $};
		\node (BR) at (2, 0) {$ B^\prime $};
		
		\draw [->] (TL) to (TM);
		\draw [->] (TR) to (TM);
		\draw [->] (BL) to (BM);
		\draw [->] (BR) to (BM);		
		
		\draw [->] (TL) to (BL);
		\draw [->] (TM) to (BM);
		\draw [->] (TR) to (BR);
	\end{tikzpicture},
\end{center}
where the outer vertical maps are $n$-connected and the middle one is $(n+1)$-connected. Then, $A \times _{C} B \rightarrow A^\prime \times _{C^\prime} B^\prime$ is $n$-connected. 
\end{lemma}

\begin{proof}
This is immediate from the long exact sequence of homotopy groups. 
\end{proof}

\begin{lemma}
\label{lemma:connectivity_of_finite_limits_of_spaces}
Let $p, q: K \rightarrow \spaces$ be diagrams of spaces indexed by a finite simplicial set $K$, and suppose that $p \rightarrow q$ is a natural transformation such that $p(k) \rightarrow q(k)$ is an $n$-connected map of spaces for each $k \in K$. Then $\varprojlim p \rightarrow \varprojlim q$ is $(n - dim(K))$-connected. 
\end{lemma}

\begin{proof}
Let us prove this by induction on the dimension of $K$, the discrete case being clear, since $n$-connected maps are stable under products. If $dim(K) > 0$, then we have a pushout diagram 

\begin{center}
	\begin{tikzpicture}
		\node (TL) at (0, 1.5) {$ \sqcup \partial \Delta^{d} $};
		\node (TR) at (1.5, 1.5) {$ \sqcup \Delta^{d} $};
		\node (BL) at (0, 0) {$ sk_{d-1}K$};
		\node (BR) at (1.5, 0) {$ K $};
		
		\draw [->] (TL) to (TR);
		\draw [->] (TL) to (BL);
		\draw [->] (TR) to (BR);
		\draw [->] (BL) to (BR);		
	\end{tikzpicture},
\end{center}
where $d = dim(K)$. By \cite{lurie_higher_topos_theory}[4.4.2.2], this induces a pullback diagram 

\begin{center}
	\begin{tikzpicture}
		\node (TL) at (0, 1.5) {$ \varprojlim p |_{\sqcup \partial \Delta^{d}} $};
		\node (TR) at (2.3, 1.5) {$ \varprojlim p |_{\sqcup \Delta^{d}}$};
		\node (BL) at (0, 0) {$ \varprojlim p |_{sk_{d-1}K} $};
		\node (BR) at (2.3, 0) {$ \varprojlim p  $};
		
		\draw [->] (TR) to (TL);
		\draw [->] (BL) to (TL);
		\draw [->] (BR) to (TR);
		\draw [->] (BR) to (BL);		
	\end{tikzpicture},
\end{center}
and analogously for $q$. We know that $\varprojlim p |_{sk_{d-1}K} \rightarrow \varprojlim q |_{sk_{d-1}K}$ and $\varprojlim p |_{\sqcup \partial \Delta^{d}} \rightarrow  \varprojlim q |_{\sqcup \partial \Delta^{d}}$ are $(n-d+1)$-connected by the inductive assumption. Moreover, since the inclusion of last vertices $\bigsqcup \Delta^{0} \hookrightarrow \bigsqcup \Delta^{d}$ is a final functor, we see that $\varprojlim p |_{\sqcup \Delta^{d}} \rightarrow \varprojlim q |_{\sqcup \Delta^{d}}$ is $n$-connected. Then, the statement follows from \textbf{Lemma \ref{lemma:connectivity_of_maps_between_pullbacks_of_spaces}}. 
\end{proof}

\begin{lemma}
\label{lemma:connectivity_on_fibres_of_undercategories}
Let $f: \ccat \rightarrow \dcat$ be a locally $n$-connected functor of $\infty$-categories and let $p: K \rightarrow \ccat$ be a diagram. Then, $\ccat _{p/} \times _{\ccat} \{ c \} \rightarrow \dcat_{fp/} \times _{\dcat} \{ f(c) \}$ is an $(n-dim(K))$-connected map of spaces for any $c \in \ccat$. 
\end{lemma}

\begin{proof}
We have an equivalence $\ccat _{p/} \times _{\ccat} \{ c \} \simeq \varprojlim _{k \in K} \map_{\ccat}(p(k), c)$ and analogously for the space $\dcat_{fp/} \times _{\dcat} \{ f(c) \}$, so that the statement is immediate from \textbf{Lemma \ref{lemma:connectivity_of_finite_limits_of_spaces}}. To see that such an identification exists, observe that by making $\ccat$ larger if necessary we can assume it has finite colimits. Then, we have 

\begin{center}
$\ccat_{p/} \times \{ c \} \simeq \map_{\Fun(K, \ccat)}(p, const(c)) \simeq \map_{\ccat}(\varinjlim _{k \in K} p(k), c) \simeq \varprojlim _{k \in K} \map_{\ccat}(p(k), c)$,
\end{center}
where we have used that representable functors take colimits to limits. 
\end{proof}

\begin{lemma}
\label{lemma:fibre_sequence_computing_maps_between_cones}
Let $p: K \rightarrow \ccat$ be a diagram indexed by a simplicial set and let $q, q^\prime: K \star \Delta^{0} \rightarrow \ccat$ be two cones under $p$ with cone points $q(\Delta^{0}) = c$ and $q^\prime(\Delta^{0}) = c^\prime$. Then, there's a fibre sequence

\begin{center}
$\map_{\ccat_{p/}}(q, q^\prime) \rightarrow \map_{\ccat}(c, c^\prime) \rightarrow \ccat_{p/} \times _{\ccat} \{ c^\prime \}$
\end{center}
of spaces, where the second map is informally given by sending $f: c \rightarrow c^\prime$ to the cone obtained by composing $q$ with $f$, and the fibre is taken over $\{ q^{\prime} \}$. 
\end{lemma}

\begin{proof}
This is done in \cite{lurie_higher_topos_theory}[5.5.5.12] in the special case of $K = \Delta^{0}$, but the proof given works equally well in general. 
\end{proof}

\begin{prop}
\label{prop:connectivity_of_functors_between_undercategories}
Let $p: K \rightarrow \ccat$ be a diagram indexed by a finite simplicial set and let $f: \ccat \rightarrow \dcat$ be a locally $n$-connected functor of $\infty$-categories. Then $\ccat_{p/} \rightarrow \dcat_{fp/}$ is locally $(n-dim(K)-1)$-connected. 
\end{prop}

\begin{proof}
Let $q, q^\prime: K \star \Delta^{0} \rightarrow \ccat$ be two cones over $p$ with cone points $c$ and $c^\prime$. Then, by \textbf{Lemma \ref{lemma:fibre_sequence_computing_maps_between_cones}} we have a commutative diagram

\begin{center}
	\begin{tikzpicture}
		\node (TL) at (0, 1) {$ \map_{\ccat_{p/}}(q, q^\prime) $};
		\node (TM) at (3, 1) {$ \map(c, c^\prime) $};
		\node (TR) at (6, 1) {$ \ccat_{/p} \times _{\ccat} \{ c^\prime \}$};
		\node (BL) at (0, 0) {$ \map_{\dcat_{fp/}}(fq, fq^\prime)$};
		\node (BM) at (3, 0) {$ \map(f(c), f(c^\prime))$};
		\node (BR) at (6, 0) {$ \dcat_{fp/} \times _{\ccat} \{ f(c^\prime) \} $};
		
		\draw [->] (TL) to (TM);
		\draw [->] (TM) to (TR);
		\draw [->] (BL) to (BM);
		\draw [->] (BM) to (BR);	

		\draw [->] (TL) to (BL);
		\draw [->] (TM) to (BM);
		\draw [->] (TR) to (BR);		
	\end{tikzpicture},
\end{center}
where the rows are fibre sequences. Since the middle map is $n$-connected by assumption, and the right one is $(n-dim(K))$-connected by \textbf{Lemma \ref{lemma:connectivity_on_fibres_of_undercategories}}, we deduce from the long exact sequence of homotopy that the left one is $(n-dim(K)-1)$-connected, which is what we wanted to show. 
\end{proof}

\section{The homotopy category of a stable $\infty$-category}

It is a result of Lurie that if $\ccat$ is a stable $\infty$-category, then its homotopy category $h \ccat$ has a canonical triangulated structure \cite[1.1.2.15]{higher_algebra}. This triangulated structure crucially depends on the stable $\infty$-category, as it is possible to find pairs $\ccat, \dcat$ of stable $\infty$-categories such that $h \ccat \simeq h\dcat$ are equivalent as categories, but not as triangulated categories.

In this appendix, we show that this triangulated structure is in fact already determined by the homotopy $2$-category. In particular, if $h_{2} \ccat \simeq h_{2}\dcat$ is an equivalence between the homotopy $2$-categories of stable $\infty$-categories $\ccat$, $\dcat$, then the induced equivalence $h \ccat \simeq h\dcat$ is triangulated. 

\begin{defin}
Let $\ccat$ be an $\infty$-category. We say that a square $\Delta^{1} \times \Delta^{1} \rightarrow \ccat$ of the form  

\begin{center}
	\begin{tikzpicture}
		\node (TL) at (0, 1.2) {$ A $};
		\node (TR) at (1.2, 1.2) {$ B $};
		\node (BL) at (0, 0) {$ C $};
		\node (BR) at (1.2, 0) {$ D $};
		
		\draw [->] (TL) to (TR);
		\draw [->] (BL) to (BR);
		\draw [->] (TL) to (BL);
		\draw [->] (TR) to (BR);		
	\end{tikzpicture},
\end{center}
is \emph{$n$-cocartesian} if for any object $Z \in \ccat$ the induced diagram of mapping spaces 

\begin{center}
	\begin{tikzpicture}
		\node (TL) at (0, 1.2) {$ \map_{\ccat}(A, Z) $};
		\node (TR) at (3, 1.2) {$ \map_{\ccat}(B, Z) $};
		\node (BL) at (0, 0) {$ \map_{\ccat}(C, Z) $};
		\node (BR) at (3, 0) {$ \map_{\ccat}(D, Z) $};
		
		\draw [->] (TR) to (TL);
		\draw [->] (BR) to (BL);
		\draw [->] (BL) to (TL);
		\draw [->] (BR) to (TR);		
	\end{tikzpicture},
\end{center}
is $n$-cartesian; in other words, that $\map_{\ccat}(D, Z) \rightarrow \map_{\ccat}(B, Z) \times_{\map_{\ccat}(A, Z)} \map_{\ccat}(C, Z)$ is an $n$-connected map of spaces. 
\end{defin}

\begin{lemma}
\label{lemma:being_cocartesian_in_stable_cat_detected_in_homotopy_2_category}
Let $\ccat$ be an $\infty$-category, $f: \Delta^{1} \times \Delta^{1} \rightarrow \ccat$ be a square and let $\pi: \ccat \rightarrow h_{2} \ccat$ denote the projection onto the homotopy $2$-category. If $f$ is cocartesian, then $\pi \circ f$ is $1$-cocartesian. If $\ccat$ is stable, then the converse also holds. 
\end{lemma}

\begin{proof}
First assume that $f$ is cocartesian. Then, for any $Z \in \ccat$, the induced diagram of mapping spaces in $h_{2} \ccat$, which by the definition of the latter is of the form 

\begin{center}
	\begin{tikzpicture}
		\node (TL) at (0, 1.2) {$ \map_{\ccat}(A, Z)_{\leq 1} $};
		\node (TR) at (3, 1.2) {$ \map_{\ccat}(B, Z)_{\leq 1} $};
		\node (BL) at (0, 0) {$ \map_{\ccat}(C, Z)_{\leq 1}$};
		\node (BR) at (3, 0) {$ \map_{\ccat}(D, Z)_{\leq 1} $};
		
		\draw [->] (TR) to (TL);
		\draw [->] (BR) to (BL);
		\draw [->] (BL) to (TL);
		\draw [->] (BR) to (TR);		
	\end{tikzpicture},
\end{center}
is $1$-cartesian, since it is obtained by $1$-truncation from a cartesian diagram. This shows that $\pi \circ f$ is $1$-cocartesian. 

To prove the converse, assume that $\ccat$ is stable and that $\pi \circ f$ is $1$-cocartesian. Then, for any $Z \in \ccat$, the induced $\pi_{0} \map_{\ccat}(D, Z) \rightarrow \pi_{0} (\map_{\ccat}(B, Z) \times _{\map_{\ccat}(A, Z)} \map_{\ccat}(C, Z))$ is an isomorphism. Replacing $Z$ by its desuspensions $\Sigma^{-k} Z$ for $k \geq 1$, we see that the same is true for the higher homotopy groups, proving that $f$ is cocartesian.
\end{proof}

\begin{lemma}
\label{lemma:1cocart_cone_exists_andunique_homotopy_class_of_maps_into_it_from_any_other_cone}
Let $\ccat$ be stable. Then, any span $p$ in $h_{2}\ccat$ can be completed to a $1$-cocartesian square. Moreover, if $q, q_{c}$ are any two completing squares and $q_{c}$ is $1$-cocartesian, then there is a unique homotopy class of maps $q_{c} \rightarrow q$ in $(h_{2}\ccat)_{p/}$. 
\end{lemma}

\begin{proof}
From the explicit construction of the homotopy $2$-category given in \cite{lurie_higher_topos_theory}[2.3.4.12], we see that if $dim(K) \leq 2$, then any diagram $K \rightarrow h_{2} \ccat$ lifts to a diagram in $\ccat$. In particular, we can lift $p$ to a span $\widetilde{p}$ in $\ccat$, which since the latter is stable can be completed to a cocartesian diagram. The image of this diagram in $h_{2} \ccat$ is then $1$-cartesian by \textbf{Lemma \ref{lemma:being_cocartesian_in_stable_cat_detected_in_homotopy_2_category}}, ending the proof of the first part. 

Now, suppose we have two cones $q, q_{c} \in (h_{2}\ccat)_{p/}$, with the latter one $1$-cocartesian, and choose lifts $\widetilde{q}, \widetilde{q_{c}} \in \ccat_{\widetilde{p}/}$. Then, since $\pi: \ccat \rightarrow h_{2}\ccat$ is locally $3$-connected, \textbf{Proposition \ref{prop:connectivity_of_functors_between_undercategories}} implies that the induced functor $\ccat_{\widetilde{p}/} \rightarrow (h_{2}\ccat)_{p/}$ is locally $1$-connected, so that 

\begin{center}
$\pi_{0} \map_{\ccat_{\widetilde{p}/}}(\widetilde{q_{c}}, \widetilde{q}) \rightarrow \pi_{0} \map_{(h_{2}\ccat)_{p/}}(q_{c}, q)$
\end{center}
is an isomorphism. It follows that it is enough to show that there is a unique homotopy class of maps $\widetilde{q_{c}} \rightarrow \widetilde{q}$, which is clear from $\textbf{Lemma \ref{lemma:being_cocartesian_in_stable_cat_detected_in_homotopy_2_category}}$, which implies that the cone $\widetilde{q}_{c}$ is cocartesian.
\end{proof}

\begin{cor}
\label{cor:1_cocartesian_cone_in_homotopy_2_cat_unique_up_to_equivalence_well_defined_up_to_htpy}
If $\ccat$ is stable, then the $1$-cocartesian cone of any span in $h_{2} \ccat$ exists and is unique up to an equivalence well-defined up to homotopy. In particular, the vertex of any such cone is unique up to an equivalence well-defined up to homotopy. 
\end{cor} 

\begin{proof}
This is immediate from \textbf{Lemma \ref{lemma:1cocart_cone_exists_andunique_homotopy_class_of_maps_into_it_from_any_other_cone}}. 
\end{proof}

\begin{rem}
By \textbf{Corollary \ref{cor:1_cocartesian_cone_in_homotopy_2_cat_unique_up_to_equivalence_well_defined_up_to_htpy}}, the $1$-cocartesian cone is weakly unique, resembling an honest colimit. However, the uniqueness of the latter is stronger - a colimit is unique up to equivalence well-defined up to a contractible space of choices, rather then just up to homotopy. 
\end{rem}

\begin{thm}
\label{thm:triangulated_structure_on_the_homotopy_cat_of_stable_cat_only_depends_on_homotopy_2_cat}
Let $\ccat$ be a stable $\infty$-category. Then, the triangulated structure on $h\ccat$ depends only on the homotopy $2$-category of $\ccat$. 
\end{thm}

\begin{proof}
We first recall Lurie's construction of the triangulated structure on the homotopy category from \cite{higher_algebra}[1.2.2.12, 1.2.2.15]. Namely, one declares a triangle 

\begin{center}
	\begin{tikzpicture}
		\node (L) at (0, 0) {$ X $};
		\node (ML) at (1.5, 0) {$ Y $};
		\node (MR) at (3, 0) {$ Z $};
		\node (R) at (4.5, 0) {$ X[1] $};
		
		\draw [->] (L) to node[auto] {$ f $} (ML);
		\draw [->] (ML) to node[auto] {$ g $} (MR);
		\draw [->] (MR) to node[auto] {$ h $} (R); 	
	\end{tikzpicture}
\end{center}
to be distinguished if there exists a diagram $\Delta^{1} \times \Delta^{2} \rightarrow \ccat$ of the form 

\begin{center}
	\begin{tikzpicture}
		\node (TL) at (0, 0) {$ X $};
		\node (TM) at (1.5, 0) {$ Y $};
		\node (TR) at (3, 0) {$ 0 $};
		\node (BL) at (0, -1) {$ 0 $};
		\node (BM) at (1.5, -1) {$ Z $};
		\node (BR) at (3, -1) {$ W $};
		
		\draw [->] (TL) to node[auto] {$ \widetilde{f} $} (TM);
		\draw [->] (TM) to (TR);
		\draw [->] (BL) to (BM);
		\draw [->] (BM) to node[auto] {$ \widetilde{h} $} (BR);
		
		\draw [->] (TL) to (BL);
		\draw [->] (TM) to node[auto] {$ \widetilde{g} $} (BM);
		\draw [->] (TR) to (BR);
	\end{tikzpicture},
\end{center}
subject to the properties that 

\begin{enumerate}
\item all three squares are cocartesian,
\item $\widetilde{f}$ represents $f$ and $\widetilde{g}$ represents $g$,
\item $h$ is the composite of the homotopy class of $\widetilde{h}$ together with the equivalence $W  \simeq X[1]$ determined by the outer square. 
\end{enumerate}

Let us consider a modified condition, which clearly depends only on the homotopy $2$-category, we will then show that it is equivalent to the one given above. Namely, let us say that a triangle 

\begin{center}
	\begin{tikzpicture}
		\node (L) at (0, 0) {$ X $};
		\node (ML) at (1.5, 0) {$ Y $};
		\node (MR) at (3, 0) {$ Z $};
		\node (R) at (4.5, 0) {$ X[1] $};
		
		\draw [->] (L) to node[auto] {$ f $} (ML);
		\draw [->] (ML) to node[auto] {$ g $} (MR);
		\draw [->] (MR) to node[auto] {$ h $} (R); 	
	\end{tikzpicture}
\end{center}
is $1$-distinguished if and only if there exists a diagram $\Delta^{1} \times \Delta^{2} \rightarrow h_{2}\ccat$ of the form given above, subject to the conditions that 

\begin{enumerate}
[label=($\arabic*^\prime$)]
\item all three squares are $1$-cocartesian,
\item $\widetilde{f}$ represents $f$ and $\widetilde{g}$ represents $g$,
\item $h$ is the composite of the homotopy class of $\widetilde{h}$ together with the equivalence $W  \simeq X[1]$ determined by the outer square,
\end{enumerate}
where we use \textbf{Corollary \ref{cor:1_cocartesian_cone_in_homotopy_2_cat_unique_up_to_equivalence_well_defined_up_to_htpy}} to make sense of the last condition. We claim that a triangle is distinguished if and only if it is $1$-distinguished, proving the theorem. 

Since $\Delta^{1} \times \Lambda^{2}_{1} \hookrightarrow \Delta^{1} \times \Delta^{2}$ is a categorical equivalence and the source is $2$-dimensional, we deduce from  \cite{lurie_higher_topos_theory}[2.3.4.12] that any diagram $\Delta^{1} \times \Delta^{2} \rightarrow h_{2}\ccat$ lifts, up to equivalence, to a diagram in $\ccat$. 

By \textbf{Lemma \ref{lemma:being_cocartesian_in_stable_cat_detected_in_homotopy_2_category}}, a square in $\ccat$ is cocartesian if and only if its image in $h_{2}\ccat$ is $1$-cocartesian, so that a diagram in $\ccat$ satisfies the condition $(1)$ if and only if its image satisfies $(1^\prime)$. Since the same is clearly true for the other two properties, we see that a diagram in the homotopy $2$-category satisfies $(1^\prime) - (3^\prime)$ if and only if all, equivalently any, of its lifts satisfy $(1) - (3)$. This shows that the property of being $1$-distinguished is equivalent to being distinguished, ending the argument.
\end{proof}

\begin{cor}
Let $\ccat, \dcat$ be stable $\infty$-categories and $h_{2}\ccat \simeq h_{2}\dcat$ be an equivalence of their homotopy $2$-categories. Then, the induced equivalence between homotopy categories is triangulated. 
\end{cor}

\begin{proof}
This is immediate from \textbf{Theorem \ref{thm:triangulated_structure_on_the_homotopy_cat_of_stable_cat_only_depends_on_homotopy_2_cat}}. 
\end{proof}

\bibliographystyle{amsalpha}
\bibliography{chromatic_homotopy_is_algebraic_bibliography}

\end{document}